\documentclass[dvipdfmx,12pt]{article}

\usepackage{amsmath}
\numberwithin{equation}{section}
\usepackage{amssymb}
\usepackage{amsthm}
\usepackage{verbatim}
\usepackage[dvipdfmx]{graphicx}
\usepackage{youngtab}
\usepackage{ytableau}

\begin{document}

\theoremstyle{definition}
\newtheorem{theorem}{Theorem}[section]
\newtheorem{definition}[theorem]{Definition}
\newtheorem{proposition}[theorem]{Proposition}
\newtheorem{lemma}[theorem]{Lemma}
\newtheorem{corollary}[theorem]{Corollary}
\newtheorem{conjecture}[theorem]{Conjecture}
\newtheorem{example}[theorem]{Example}
\newtheorem{remark}[theorem]{Remark}

\ytableausetup{centertableaux, mathmode, boxsize=1.5em}


\title{Multiple Hook Removing Game Whose Starting Position is a Rectangular Young Diagram with the Unimodal Numbering}
\author{Tomoaki Abuku\footnote{National Institute of Informatics, 2-1-2 Hitotsubashi, Chiyoda-ku, Tokyo 101-8430, Japan \newline e-mail: buku3416@gmail.com}\qquad Masato Tada\footnote{University of Tsukuba, 1-1-1 Tennodai, Tsukuba, Ibaraki 305-8571, Japan \newline e-mail: t-d-masato@math.tsukuba.ac.jp}}
\date{2021}

\maketitle


\begin{abstract}
We introduce a new impartial game, named Multiple Hook Removing Game (MHRG for short).
We also determine the $\mathcal{G}$-values of some game positions (including the starting positions) in MHRG$(m,n)$, the MHRG whose starting position is the rectangular Young diagram of size $m\times n$ with the unimodal numbering.
In addition, we prove that MHRG$(m,n)$ is isomorphic, as games, to MHRG$(m,n+1)$ (if $m\le n$ and $m+n$ is even), and give a relationship between MHRG$(n,n+1)$ (and MHRG$(n,n)$) and HRG$(S_n)$, the Hook Removing Game in terms of shifted Young diagrams.
\end{abstract}


\section{Introduction}
Hook Removing Game (HRG for short) is an impartial game whose game positions are Young diagrams.
HRG is played by two players, and each players alternately remove one hook from the Young diagram.
It was invented in 1970 by Mikio Sato (see, \cite{Sato1} and \cite{Sato2}), who also found a formula for the $\mathcal{G}$-values.
On the other hand, as is well-known, Young diagrams are used also in combinatorial representation theory.
For example, there exists one-to-one correspondence between the (isomorphism classes of) irreducible representation for the symmetric group $S_n$ and the Young diagrams having $n$ boxes.
Under this correspondence, the dimension of an irreducible representation for $S_n$ is equal to the number of standard Young tableaux of the corresponding shape.
Moreover, Irie \cite{Irie} gave an interesting relation between the $\mathcal{G}$-values in HRG and the representation theory of $S_n$.
In \cite{Proctor}, Proctor introduced the notion of a $d$-complete poset, which generalizes the Young diagrams from the viewpoint of the combinatorial theory, and is used for an expression of an element in a quotient of the Weyl group for a simply-laced Kac-Moody Lie algebra (including the finite-dimensional simple Lie algebra of type A,D, and E).
In particular, $d$-complete posets of ``shape" (resp., ``shifted shape") type correspond to Young diagrams (resp., shifted Young diagrams), and give an expression of an element in the quotient of the Weyl group of type A (resp., type D).  
In 2001, Kawanaka \cite{Kawanaka} introduced the notion of Plain Game, which generalizes HRG;
$d$-complete posets are used instead of Young diagrams.
He also gave a closed formula for the $\mathcal{G}$-values in Plain Game.

In this paper, we introduce a new impartial game, named Multiple Hook Removing Game (MHRG for short).
The rules of MHRG (which will be explained in \S 1.1 below) are motivated by Tada's work \cite{Tada}, in which he described the elements in a minimal parabolic quotient of the Weyl group of type B,C,F,G in terms of $d$-complete posets with a ``coloring" (or ``numbering");
this ``coloring" naturally arises from the fact that the Dynkin diagrams of type B,C,F,G are obtained from the ones of type A,D,E by ``folding".
In particular, Young diagrams contained in a rectangular Young diagrams with the ``unimodal numbering" (see Example \ref{ex:MHRG} below) correspond to elements in the Weyl group of type B,C.

\begin{table}[htbp]
\centering
\begin{tabular}{|c|c|c|} \hline
    Game &Diagram &\begin{tabular}{c}Corresponding\\ Weyl group\end{tabular} \\ \hline\hline
    \begin{tabular}{c}Sato-Welter Game\end{tabular}
    &\begin{tabular}{c}Young diagram\end{tabular}
    & type A \\ \hline
    \begin{tabular}{c}Turning Turtles\end{tabular} 
    &\begin{tabular}{c}shifted Young diagram\end{tabular}
    & type D\\ \hline
    \begin{tabular}{c}Plain Game \end{tabular}
    & \begin{tabular}{c}$d$-complete poset\end{tabular}
    & type A,D,E\\ \hline
    \begin{tabular}{c}{\bf Multiple Hook}\\{\bf Removing Game}\end{tabular} 
    & \begin{tabular}{c}{\bf Young diagram with the}\\{\bf unimodal numbering}\end{tabular} 
    & {\bf type B,C}\\ \hline
\end{tabular}
\end{table}
\noindent
In this paper, as explained below, we will prove some results on $\mathcal{G}$-values in MHRG with the ``unimodal numbering".


\subsection{Rules of Multiple Hook Removing Game}
In this subsection, we explain the rules of MHRG (in a general setting).

\begin{definition}\label{MHRG}
The rules of MHRG are as follows (see also Sections 2.3 and 2.4):
\begin{enumerate}
    \item[(M1)] The game is played by two players.
    \item[(M2)] The two players alternately make a move.
    \item[(M3)] The starting position is a Young diagram $Y_s$ with a numbering $\alpha:Y_s\to\mathbb{N}$.
    All game positions are Young diagrams $Y$ contained in $Y_s$ with a numbering $\alpha\vert_Y$.
    \item[(M4)] Given a Young diagram $Y$ with the numbering $\alpha\vert_Y$, the player chooses a box in $Y$, and remove the hook $h$ corresponding to the box. 
    Let $\mathcal{A}(h)$ be the multiset of the numbers (in boxes) in the hook $h$, and let $Y'$ be the Young diagram obtained by removing $h$ from $Y$, with the numbering $\alpha\vert_{Y'}$.
    \begin{enumerate}
        \item[(M4a)] If there does not exist any box in $Y'$ whose corresponding hook $h'$ satisfies $\mathcal{A}(h') = \mathcal{A}(h)$ as multisets, then the player's turn is over, and the next player is given $Y'$.
        \item[(M4b)] If there exists a box in $Y'$ whose corresponding hook $h'$ satisfies $\mathcal{A}(h') = \mathcal{A}(h)$ as multisets, then the player must choose one of such boxes, and remove the hook $h'$ corresponding to the box. 
        Let $Y''$ be the Young diagram obtained by removing $h'$ from $Y'$, with the numbering $\alpha\vert_{Y''}$.
        \item[(M4c)] Do the same operation as (M4a) and (M4b), with $Y'$ replaced by $Y''$.
        As long as such a box exists, repeat this operation.
    \end{enumerate}
    \item[(M5)] The player who makes the empty Young diagram $\emptyset$ wins.
\end{enumerate}
\end{definition}

In this paper, we introduce a special numbering (which we call the unimodal numbering) for the boxes in Young diagrams; see Section 3.1 below.
Then, we mainly treat MHRG$(m,n)$ for $m,n\in \mathbb{N}$ which is MHRG whose starting position $Y_s$ is the rectangular Young diagram $Y_{m,n}$ of size $m\times n$ with the unimodal numbering $\alpha_{m,n}$.

\begin{example}\label{ex:MHRG}
At the beginning of MHRG$(3,5)$ played by A and B, the following Young diagram $Y = Y_{3,5}$ with the numbering $\alpha_{3,5}$ is given to the player, say A, having the first move, as the starting position. 
\begin{center}
$Y=$
\begin{ytableau}
    3&4&3&2&1\\
    2&3&4&3&2\\
    1&2&3&4&3
\end{ytableau}
\end{center}
If the player A removes the hook $h$ corresponding to the box $(2,4)$ from $Y$, then A obtains $Y'$ (with $\alpha_{3,5}\vert_{Y'}$) below:
\begin{center}
$Y=$
\begin{ytableau}
    3&4&3&2&1\\
    2&3&4&*(gray)3&*(gray)2\\
    1&2&3&*(gray)4&3
\end{ytableau}
\qquad$\xrightarrow{}$\qquad
$Y'=$
\begin{ytableau}
    3&4&3&2&1\\
    2&3&4&3\\
    1&2&3
\end{ytableau}
\end{center}
Note that $\mathcal{A}(h)=\{2,3,4\}$.
Since there does not exist a box in $Y'$ whose corresponding hook $h'$ satisfies $\mathcal{A}(h')=\mathcal{A}(h)=\{2,3,4\}$, the player A's turn is over.
If the player B removes the hook $h'$ corresponding to the box $(2,1)$ from $Y'$, then B obtains $Y''$ (with $\alpha_{3,5}\vert_{Y''}$) below:
\begin{center}
$Y'=$
\begin{ytableau}
    3&4&3&2&1\\
    *(gray)2&*(gray)3&*(gray)4&*(gray)3\\
    *(gray)1&2&3
\end{ytableau}
\qquad$\xrightarrow{}$\qquad 
$Y''=$
\begin{ytableau}
    3&4&3&2&1\\
    2&3
\end{ytableau}
\end{center}
Note that $\mathcal{A}(h')=\{3,4,3,2,1\}$.
Notice that the box $(1,2)$ in $Y''$ is a unique box in $Y''$ whose corresponding hook $h''$ satisfies $\mathcal{A}(h'')=\mathcal{A}(h')=\{3,4,3,2,1\}$.
Because of (M4b), B must remove the hook $h''$ from $Y''$, and obtains $Y'''$ (with $\alpha_{3,5}\vert_{Y'''}$) below:

\begin{center}
$Y''=$
\begin{ytableau}
    3&*(gray)4&*(gray)3&*(gray)2&*(gray)1\\
    2&*(gray)3
\end{ytableau}
\qquad$\xrightarrow{}$\qquad
$Y'''=$
\begin{ytableau}
    3\\
    2
\end{ytableau}
\end{center}
If the player A removes the hook $h'''$ corresponding to the box $(1,1)$ from $Y'''$, then A obtains empty Young diagram $\emptyset$:
\begin{center}
$Y'''=$
\begin{ytableau}
    *(gray)3\\
    *(gray)2
\end{ytableau}
\qquad$\xrightarrow{}$\qquad
$\emptyset$
\end{center}
In this case, the winner is the player A.
Remark that $Y''$ above does not appear in the game tree of MHRG$(3,5)$; 
in general, not all Young diagrams contained in $Y_{m,n}$ is a position of MHRG$(m,n)$ (see Remark \ref{re:t_not=_f}).
After this paper, Motegi \cite{Motegi} gave a characterization of the set of all game positions in MHRG$(m,n)$.
\end{example}


\subsection{Main Results}

By using computer, we obtained the $\mathcal{G}$-value of the starting position in MHRG$(m,n)$ for $1 \le m,n \le 9$ as Table \ref{table:starting}.

\begin{table}[htbp]
\centering
\begin{tabular}{|c||c|c|c|c|c|c|c|c|c|} \hline
    $m$\textbackslash $n$&1&2&3&4&5&6&7&8&9\\ \hline\hline
    1&1&1&3&3&5&5&7&7&9 \\ \hline
    2&1&3&3&1&1&1&1&1&1\\ \hline
    3&3&3&0&0&0&0&3&3&10 \\ \hline
    4&3&1&0&4&4&2&2&5&5 \\ \hline
    5&5&1&0&4&1&1&14&14&18 \\ \hline
    6&5&1&0&2&1&7&7&0&0 \\ \hline
    7&7&1&3&2&14&7&0&0&10\\\hline
    8&7&1&3&5&14&0&0&8&8\\ \hline
    9&9&1&10&5&18&0&10&8&1 \\ \hline
\end{tabular}
\caption{$\mathcal{G}$-value of the starting position $Y_{m,n}$ in MHRG$(m,n)$ for $1 \le m,n \le 9$.}\label{table:starting}
\end{table}

Motivated by this table, we made conjectures on the $\mathcal{G}$-value of the starting position $Y_{m,n}$ of MHRG$(m,n)$ for some $m,n$ as follows:
\begin{enumerate}
    \item [(1)] If $m\leq n$ and $m+n$ is even, then the $\mathcal{G}$-value of the starting position in MHRG$(m,n)$ is equal to the $\mathcal{G}$-value of the starting position in MHRG$(m,n+1)$.
    \item [(2)] The sequence $\{\mathcal{G}(Y_{1,n})\}_{n\ge 1}$ of the $\mathcal{G}$-values of the starting positions in MHRG$(1,n)$ for $n\ge 1$ is arithmetric periodic.
    \item [(3)] The sequence $\{\mathcal{G}(Y_{2,n})\}_{n\ge 2}$ of the $\mathcal{G}$-values of the starting positions in MHRG$(2,n)$ for $n\ge 2$ is periodic.
    \item [(4)] The $\mathcal{G}$-value of the starting position in MHRG$(n,n)$ and MHRG$(n,n+1)$ is equal to $\bigoplus_{1\leq k\leq n}^{} k$, where $\bigoplus_i^{}{a_i}$ denotes the nim-sum (the addition of numbers in binary form without carry) of all $a_i$'s.
\end{enumerate}

In fact, we prove the following theorems, which implies that our conjectures above are true.
Let $\mathcal{F}(Y_{m,n})$ denote the set of all Young diagrams contained in $Y_{m,n}$, and $\mathcal{T}(Y_{m,n})$ be the subset of $\mathcal{F}(Y_{m,n})$ consisting of all $Y\in\mathcal{F}(Y_{m,n})$ such that there exists a transition from $Y_{m,n}$ to $Y$.

\begin{theorem}[= Theorem \ref{th:game_cor}]\label{th:mnI}
Let $m,n\in\mathbb{N}$ be such that $m\le n$ and $m+n$ is even.
Then the map $E$ (see Definition \ref{def:index_cor}) gives an isomorphism from MHRG$(m,n)$ to MHRG$(m,n+1)$.
Therefore, for each $Y\in \mathcal{T}(Y_{m,n})$, it holds that $\mathcal{G}(Y) = \mathcal{G}(E(Y))$. 
In particular, $\mathcal{G}(Y_{m,n})$ in MHRG$(m,n)$ is equal to $\mathcal{G}(Y_{m,n+1})$ in MHRG$(m,n+1)$.
\end{theorem}

\begin{theorem}[= Theorem \ref{th:1n}]\label{th:1nI}
Let $m=1$, and $n\in \mathbb{N}$.
In MHRG$(1,n)$, 
\begin{center}
$\mathcal{T}(Y_{1,n})=\left \{ 
\begin{array}{cc}
    \mathcal{F}(Y_{1,n})  & \text{if $n$ is odd},\\ 
    \mathcal{F}(Y_{1,n})\setminus\{(\frac{n}{2})\} & \text{if $n$ is even}.
\end{array}
\right.$
\end{center}
Moreover, for $0\leq l\leq n$ such that $Y_{1,l}\in \mathcal{T}(Y_{1,n})$,
\begin{center}
$\mathcal{G}((l))=\left \{ 
\begin{array}{ccc}
    l  & \text{if $n$ is odd},\\ 
    l & \text{if $n$ is even and $l<n/2$},\\
    l-1 & \text{if $n$ is even and $n/2<l$}.
\end{array}
\right.$
\end{center}
In particular,
\begin{center}
$\mathcal{G}((n))=\left \{ 
\begin{array}{cc}
    n  & \text{if $n$ is odd},\\ 
    n-1 & \text{if $n$ is even}.
\end{array}
\right.$
\end{center}
\end{theorem}

\begin{theorem}[see Lemma \ref{lemma:2n_transition}, Theorem \ref{th:2n}, and Corollary \ref{cor:2n+1}]\label{th:2nI}
Let $m=1$, and $n\ge 2$. 
If $n$ is even, then
\[
\mathcal{T}(Y_{2,n}) = \mathcal{F}(Y_{2,n})\setminus\{(\lambda'_1,\lambda'_2) \in \mathcal{F}(Y_{2,n})\mid\lambda'_1+\lambda'_2 = n\}.
\]
Moreover, assume that $n$ is even.
For $Y = (\lambda_1,\lambda_2) \in \mathcal{T}(Y_{2,n})$, if $\lambda_1 + \lambda_2 < n$, then $\mathcal{G}(Y)$ is equal to the $\mathcal{G}$-value of the game position corresponding to $Y$ in the Sato-Welter game, and the list of those $Y = (\lambda_1,\lambda_2) \in \mathcal{F}(Y_{2,n})$ with $\lambda_1 + \lambda_2 > n$ whose $\mathcal{G}$-values are $0$,$1$ or $2$ is given by Table \ref{table:012u} (in Section 4.4).
In particular, for $n\ge 2$, the $\mathcal{G}$-value of the starting position $Y_{2,n}$ is as follows: 
\begin{equation*}
\mathcal{G}(Y_{2,n}) = \left\{
\begin{aligned}
3\ &(n = 2,3),\\
2\ &(n = 2+8m, 3+8m\ (m\in\mathbb{N})),\\
1\ &(\text{otherwise}).
\end{aligned}\right.
\end{equation*}
\end{theorem}

\begin{theorem}[see Theorem \ref{th:nn+1_shifted} and Corollary \ref{cor:nn_nn+1}]\label{th:nnI}
For $n\in\mathbb{N}$, MHRG$(n,n+1)$ (and MHRG$(n,n)$) is isomorphic to HRG$(S_n)$; where HRG$(S_n)$ is HRG whose starting position is the shifted Young diagram $S_n = (n,n-1,\ldots,1)$.
In particular, $\mathcal{G}(Y_{n,n})$ in MHRG$(n,n)$ and $\mathcal{G}(Y_{n,n+1})$ in MHRG$(n,n+1)$ is equal to $\mathcal{G}(S_n)=\bigoplus_{1\leq k\leq n}^{} k$ in HRG$(S_n)$.
\end{theorem}

This paper is organized as follows. 
In Section 2, we fix our notation for Young diagrams and combinatorial game theory. 
In Section 3, we introduce the unimodal numbering, and explain the diagonal expression for Young diagrams. 
In Section 4, we prove Theorems \ref{th:mnI}, \ref{th:1nI}, and \ref{th:2nI} above. 
In Sections 5, we explain the diagonal expression for shifted Young diagrams, and prove Theorem \ref{th:nnI} above.

\paragraph{Acknowledgements.}
The authors would like to thank Professor Daisuke Sagaki and Yuki Motegi for their advice.


\section{Preliminaries}

We denote by $\mathbb{N}_{0}$ the set of all non-negative integers, and $\mathbb{N}$ the set of all positive integers.

\subsection{Combinatorial Games}

MHRG is an impartial game in the following sense (for the details on the combinatorial game theory, see, e.g., \cite{Conway} and \cite{Siegel}).
\begin{itemize}
    \item\ The game is played by two players.
    \item\ Two players alternately make a move.
    \item\ The player who makes the last move wins. 
    \item\ No chance elements (the possible moves in each position are determined from the beginning).
    \item\ No hidden information (both players will have complete knowledge of the game state at all times).
    \item\ All game positions are ``short'' (namely, there are finitely many positions that can be reached from a position, and any position can not appear twice in a play).
    \item\ Both players have the same set of possible moves in each position. 
\end{itemize}

\begin{definition}\label{def:outcome_classes}
A game position is called an $\mathcal{N}$-position (resp., $\mathcal{P}$-position) if the next (resp., previous) player has a winning strategy.
\end{definition}

\begin{proposition}[{\cite[Chapter I]{Siegel}}]
Each impartial game position is either an $\mathcal{N}$-position or a $\mathcal{P}$-position.
\end{proposition}

\begin{proposition}[{\cite[Chapters I and IV]{Siegel}}]
Let $G$ be an impartial game position.
If $G$ is an $\mathcal{N}$-position, then there exists a move from $G$ to a $\mathcal{P}$-position.
If $G$ is a $\mathcal{P}$-position, then there exists no move from $G$ to a $\mathcal{P}$-position.
\end{proposition}

For an impartial game $\mathcal{X}$, we set $\mathcal{C}(\mathcal{X})$ the all game positions of $\mathcal{X}$.
Let $G,G' \in \mathcal{C}(\mathcal{X})$.
If $G'$ can be reached from $G$ by a single move, then we say that $G'$ is an option of $G$, and we write $G \rightarrow G'$ .
We set $\mathcal{O}(G) := \{G' \in \mathcal{C}(\mathcal{X}) \mid G \rightarrow G'\}$.
A transition from $G$ to $G'$ is, by definition, a sequence $G = G_0,G_1,\ldots,G_k=G', k\in\mathbb{N}_0,$ of game positions in $\mathcal{C}(\mathcal{X})$ such that
\[
G = G_0\xrightarrow{}G_1\xrightarrow{}\cdots\xrightarrow{} G_k=G'.
\]

\begin{definition}\label{def:game_iso}
Let $\mathcal{X}$ and $\mathcal{Y}$ be impartial games.
If there exists a bijection $f:\mathcal{C}(\mathcal{X})\to\mathcal{C}(\mathcal{Y})$ such that $\mathcal{O}(f(G)) = f(\mathcal{O}(G))$ for $G \in \mathcal{C}(\mathcal{X})$, then we say that $\mathcal{X}$ is isomorphic to $\mathcal{Y}$, and we call $f$ an isomorphism from $\mathcal{X}$ to $\mathcal{Y}$.
\end{definition}


\subsection{$\mathcal{G}$-value}
We recall the $\mathcal{G}$-value of a position in an impartial game.

\begin{definition}\label{def:G_value}
Let $G$ be a game position.
We define $\mathcal{G}(G)\in \mathbb{N}_{0}$, called the $\mathcal{G}$-value of $G$, by
\[
\mathcal{G}(G):=\mathrm{mex} \{\mathcal{G}(G')\mid G \rightarrow G'\},
\]
where for a proper subset $T$ of $\mathbb{N}_{0}$, we set $\mathrm{mex}\ T:=\mathrm{min} (\mathbb{N}_0 \setminus T).$

\end{definition}

The following theorem is well-known.

\begin{theorem}[{\cite[Chapter IV]{Siegel}}]\label{th:G_value}
For a game position $G$, $\mathcal{G}(G)=0$ if and only if $G$ is a $\mathcal{P}$-position.
\end{theorem}

By Theorem \ref{th:G_value}, if we want only to determine whether a given game position $G$ is a $\mathcal{P}$-position or not, we need only to check if $\mathcal{G}(G)=0$ or not;
we do not need the precise value of $\mathcal{G}(G)$.
However, it is worthwhile to obtain a closed formula for $\mathcal{G}(G)$, because there is a nice application to, for example, the disjunctive sums of game positions;
see, e.g., \cite[Chapter IV]{Siegel}.

The following proposition can be easily shown.

\begin{proposition}\label{prop:game_iso}
Let $\mathcal{X}$ and $\mathcal{Y}$ be games.
If there exists an isomorphism $f:\mathcal{C}(\mathcal{X})\to\mathcal{C}(\mathcal{Y})$, then $\mathcal{G}(G)=\mathcal{G}(f(G))$ for $G \in \mathcal{C}(\mathcal{X})$. 
\end{proposition}


\subsection{Young Diagrams and Numbering}

Let $m\in \mathbb{N}$, and let $\lambda_1,\ldots,\lambda_m\in\mathbb{N}_{0}$ be such that $\lambda_1\ge \cdots\ge\lambda_m\ge 0$.
Then, the set $Y=(\lambda_1,\ldots,\lambda_m):= \{(i,j)\in \mathbb{N}^2\mid 1\le i \le m, 1\le j \le \lambda_i\}$ is called the Young diagram corresponding to $(\lambda_1,\ldots,\lambda_m)$.
An element of the Young diagram is called a box, and the Young diagram is described in terms of boxes as follows.

\ytableausetup{centertableaux, mathmode, boxsize=2.5em}
\begin{center}
$Y=(6,6,5,3,3)=$
\begin{ytableau}
    (1,1)&(1,2)&(1,3)&(1,4)&(1,5)&(1,6)\\
    (2,1)&(2,2)&(2,3)&(2,4)&(2,5)&(2,6)\\
    (3,1)&(3,2)&(3,3)&(3,4)&(3,5)\\
    (4,1)&(4,2)&(4,3)\\
    (5,1)&(5,2)&(5,3)\\
\end{ytableau}
\end{center}

For $i\in \mathbb{N}$, the subset $\{(i,j)\mid j\in \mathbb{N}\}\cap Y$ of $Y$ is called the $i$-th row of $Y$.
Similarly, for $j\in \mathbb{N}$, the subset $\{(i,j)\mid i\in \mathbb{N}\}\cap Y$ of $Y$ is called the $j$-th column of $Y$.

For a Young diagram $Y$, let $\mathcal{F}(Y)$ denote the set of all Young diagrams contained in $Y$.
Let $\#(Y)$ denote the number of the boxes contained in $Y$.
It is obvious that if $Y' \subseteq Y$, then $\#(Y') \le \#(Y)$.

For a Young diagram $Y$, a map $\alpha: Y \to \mathbb{N}$ is called a numbering of $Y$.
For a box $(i,j)\in Y$, if $\alpha(i,j) = x$, then we say that the box $(i,j)$ has the number $x$.
Let $Y$ be a Young diagram with a numbering $\alpha$.
For a subset $X$ of $Y$, we set $\mathcal{A}(X) = \mathcal{A}_{\alpha}(X) := \{\!\{ \alpha(i,j)\mid (i,j)\in X\}\!\}$, where $\{\!\{x_1,\ldots,x_N \}\!\}$ denotes the multiset consisting of $x_1,\ldots,x_N$.


\subsection{Hooks of a Young Diagram}

\begin{definition}\label{def:hook}
For a box $(i,j)$ of a Young diagram $Y$,
\[
h(i,j)=h_Y{(i,j)} := \{(i,j)\} \sqcup \{(i',j)\in Y\mid i' > i\} \sqcup \{(i,j')\in Y\mid j' > j\}
\]
is called the hook (in $Y$) corresponding to the box $(i,j)$.
\end{definition}

\begin{definition}\label{def:hook_removing}
For a box $(i,j)$ of a Young diagram $Y$, we remove the hook $h_Y{(i,j)}$ corresponding to the box $(i,j)$ as follows (see also Example \ref{ex:hook_removing} below):
\begin{enumerate}
\item Remove each box in the hook $h_Y{(i,j)}$.
\item Move each box $(i',j')$ satisfying $i'>i$ and $j'>j$ to $(i'-1,j'-1)$.
\end{enumerate}
We denote by $Y\setminus h_Y{(i,j)}$ the Young diagram obtained by removing the hook $h_Y{(i,j)}$ corresponding to the box $(i,j)$ from $Y$.
\end{definition}

\begin{example}\label{ex:hook_removing}
If we remove the hook corresponding to the box $(2,2)$ from the Young diagram $Y = (6,6,5,3,3)$, then we get $Y'=Y\setminus h_Y{(2,2)}=(6,4,2,2,1)$.\\
\end{example}

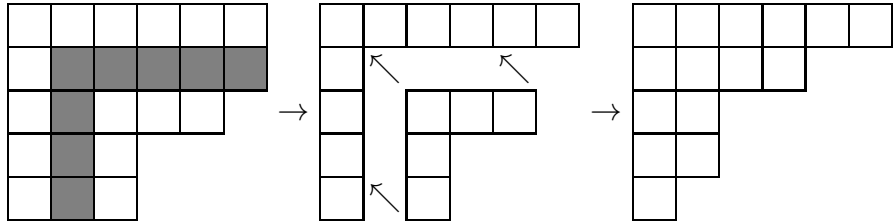
\ytableausetup{centertableaux, mathmode, boxsize=1.35em}
\begin{center}
\begin{ytableau}
    \ &\ &\ &\ &\ &\ \\
    \ &*(gray)\ &*(gray)\ &*(gray)\ &*(gray)\ &*(gray)\ \\
    \ &*(gray)\ &\ &\ &\ \\
    \ &*(gray)\ &\ \\
    \ &*(gray)\ &\ \\
\end{ytableau}
$\to$
\begin{ytableau}
    \ &\ &\ &\ &\ &\ \\
    \ &\none[\nwarrow] &\none &\none &\none[\nwarrow] &\none \\
    \ &\none &\ &\ &\ \\
    \ &\none &\ \\
    \ &\none[\nwarrow] &\ \\
\end{ytableau}
$\to$
\begin{ytableau}
    \ &\ &\ &\ &\ &\ \\
    \ &\ &\ &\ \\
    \ &\ \\
    \ &\ \\
    \ \\
\end{ytableau}\\
\end{center}


\section{Unimodal Numbering and Diagonal Expression}


\subsection{Unimodal Numbering of Young Diagrams}

Let and fix $m,n\in\mathbb{N}$.
We denote by $Y_{m,n}:= \{(i,j)\in \mathbb{N}^2\mid 1\le i \le m, 1\le j \le n\}$ the rectangular Young diagram.
For $Y\in \mathcal{F}(Y_{m,n})$, we define a special numbering $\alpha_{m,n}:Y \to \mathbb{N}$, called the unimodal numbering of $Y$, as follows:
For $(i,j)\in Y$, we set $\alpha_{m,n}(i,j):=\min\{j-i+m, i-j+n\} \in \mathbb{N}$.
In what follows, boxes in $Y\in \mathcal{F}(Y_{m,n})$ are always numbered by the unimodal numbering $\alpha_{m,n}$.

\begin{figure}[h]
\centering
\begin{ytableau}
    3&4&3&2&1\\
    2&3&4&3&2\\
    1&2&3&4&3
\end{ytableau}
\qquad
\begin{ytableau}
    4&5&5&4&3&2&1\\
    3&4&5&5&4&3&2\\
    2&3&4&5&5&4&3\\
    1&2&3&4&5&5&4
\end{ytableau}
\caption{unimodal numberings in the case that $m=3$ and $n=5$, and the case that $m=4$ and $n=7$.}
\end{figure}

\begin{remark}\label{remark:unimodal}
Let $Y\in \mathcal{F}(Y_{m,n})$.
By the definition of the unimodal numbering $\alpha_{m,n}$, we can easily check the following.
\begin{enumerate}
    \item[(1)] If $Y$ contains the box $(m,1)$, then it has the number $1$.
    If $Y$ contains the box $(1,n)$, then it has the number $1$.
    \item[(2)] The boxes $(i,j)$ and $(i+1,j+1)$ have the same number (if they exist in $Y$).
    \item[(3)] The maximum value of $\alpha_{m,n}:Y_{m,n} \to \mathbb{N}$ is equal to $\hat{\alpha}_{m,n} := \lfloor (n+m)/2 \rfloor$, where $\lfloor x \rfloor := \max\{y\in\mathbb{Z}\mid y \le x\}$ for $x\in \mathbb{R}$.
\end{enumerate}
\end{remark}


\subsection{Diagonal Expression of a Young Diagram}

We define the diagonal expression of a Young diagram in $\mathcal{F}(Y_{m,n})$.
We write an element $\boldsymbol{a} \in \mathbb{N}_0^{m+n+1}$ as $\boldsymbol{a} = (a_{-m},a_{-m+1},\ldots,a_{-1},\dot{a_{0}},a_{1},\ldots,a_{n-1},a_n)$, and we call $a_{k}$ the $k$-th component of $\boldsymbol{a}$ for $-m \le k \le n$.
If necessary, we will accentuate the $0$-th component (i.e., $a_0$ in $\boldsymbol{a}$ above) by putting a dot above it as above.

\begin{definition}\label{def:AD_condition}
Let $\boldsymbol{a}=(a_{-m},\ldots,a_{n}) \in \mathbb{N}_0^{m+n+1}$.
For $-m < i \le 0$ (resp., $0 < i \le n$), we say that the pair $(a_{i-1}, a_{i})$ satisfies the adjacency condition if $0\le a_i-a_{i-1}\le 1$ (resp., $0\le a_{i-1}-a_i\le 1$).
We say that $\boldsymbol{a}$ satisfies the adjacency condition if $(a_{i-1}, a_{i})$ satisfies the adjacency condition for all $-m < i \le n$.
\end{definition}

For $m,n\in\mathbb{N}$, let $\mathbb{D}_{m,n}\subset\mathbb{N}_0^{m+n+1}$ denote the set of all element $\boldsymbol{a}=(a_{-m},\ldots,a_{n}) \in \mathbb{N}_0^{m+n+1}$ with $a_{-m} = a_n = 0$ satisfying the adjacency condition.

\begin{example}\label{ex:DSeq}
Let $m = 2$ and $n = 4$.
Then,
\[
(0,1,\dot{2},1,1,1,0) \in \mathbb{D}_{2,4},\qquad (0,0,\dot{2},2,1,1,0) \notin \mathbb{D}_{2,4}.
\]
In the second case, the pair $(a_{-1}, a_0) = (0,2)$ does not satisfy the adjacency condition.
\end{example}

Let $Y \in \mathcal{F}(Y_{m,n})$, and set $d_k = d_k(Y) := \#\{(i,j)\in Y\mid j-i=k\}$ for $k \in \mathbb{Z}$.
Note that if $k \le -m$ or $k \ge n$, then $d_k = 0$.

\begin{remark}\label{remark:d_box}
For $i,j\geq 2$, if $(i,j)\in Y$, then $(i-1,j-1) \in Y$.
Also, if $(i,j)\notin Y$, then $(i+a,j+a) \notin Y$ for $a\in\mathbb{N}$.
Hence we see that $d_k = \max\{\min\{i,j\} \mid (i,j)\in Y, j-i=k\}$ for $k\in \mathbb{Z}$.
\end{remark}

We can easily show the following lemma, which will be used later.

\begin{lemma}\label{lemma:diagonal}
Keep the notation and setting above. 
\begin{enumerate}
    \item[(1)] Let $k\ge0$.
    Then, $(d_k+1,d_k+k+1)\notin Y$.
    Moreover, if $d_k > 0$, then $(d_k,d_k+k)\in Y$.
    \item[(2)] Let $k<0$.
    Then $(d_k-k+1,d_k+1)\notin Y$.
    Moreover, if $d_k > 0$, then $(d_k-k,d_k)\in Y$.
\end{enumerate}
\end{lemma}

The following proposition is well-known for experts;
in Appendix A, we give a proof for it for completion.

\begin{proposition}\label{prop:DRep}
For every $Y \in \mathcal{F}(Y_{m,n})$, 
\[
\boldsymbol{d}(Y) = \boldsymbol{d}_{m,n}(Y) := (d_{-m}(Y),\ldots,d_{n}(Y)) \in \mathbb{D}_{m,n}.
\]
Moreover, the map 
\[
\boldsymbol{d} = \boldsymbol{d}_{m,n}:\mathcal{F}(Y_{m,n})\to \mathbb{D}_{m,n}, Y \mapsto \boldsymbol{d}(Y) = \boldsymbol{d}_{m,n}(Y),
\]
is a bijection;
we call $\boldsymbol{d}(Y)$ the diagonal expression of $Y$.
\end{proposition}

\begin{example}\label{ex:DRep}
Assume that $m=3$ and $n=5$.
If $Y \in \mathcal{F}(Y_{3,5})$ is
\begin{center}
$Y = $
\begin{ytableau}
    3&4&3&2&1\\
    2&3&4&3\\
    1&2&3
\end{ytableau}\ ,
\end{center}
then $\boldsymbol{d}(Y) = \boldsymbol{d}_{3,5}(Y) = (0,1,2,\dot{3},2,2,1,1,0)$.
\end{example}


\subsection{Diagonal Expression and Hooks}
Let $Y\in\mathcal{F}(Y_{m,n})$.
Let $(i,j)\in Y$, and set $Y' := Y\setminus h_Y{(i,j)}$.
Let $(i',j)$ be the bottom box in the $j$-th column of $Y$ (note that $i' \ge i$), and $(i,j')$ the rightmost box in the $i$-th row of $Y$ (note that $j' \ge j$).
For $k \in \mathbb{Z}$, it holds that
\[
d_k(Y)-d_k(Y')=
\begin{cases}
    1 & \text{if } j-i'\le k\le j'-i,\\
    0 & \text{otherwise}.
\end{cases}
\]
Hence the diagonal expression of $Y'$ is
\begin{align*}
\boldsymbol{d}(Y') = (\ldots,d_{j-i'-1}(Y),\ &d_{j-i'}(Y)-1,d_{j-i'+1}(Y)-1,\ldots,\\
&d_{j'-i-1}(Y)-1,d_{j'-i}(Y)-1,d_{j'-i+1}(Y),\ldots). \tag{3.1} \label{eq:1}
\end{align*}

Let $\boldsymbol{a} = (a_{-m},\ldots, a_n) \in \mathbb{D}_{m,n}$ , $\boldsymbol{a'} = (a'_{-m},\ldots, a'_n) \in \mathbb{N}_0^{m+n+1}$, and $-m < l \le r < n$.
If $a'_{k} = a_{k}-1$ for $l \le k \le r$, and $a'_{k} = a_{k}$ for the other $k$'s, then we write $\boldsymbol{a}\xrightarrow{l,r}\boldsymbol{a'}$ or $\boldsymbol{a'} = \boldsymbol{a}_{[l,r]}$.
In this case, the pair $(a'_{k-1}, a'_{k})$ satisfies the adjacency condition for all $-m<k\le n$ but $k=l,r+1$.
Hence, if $(a'_{l-1}, a'_{l})$ and $(a'_{r}, a'_{r+1})$ satisfy the adjacency condition, then $\boldsymbol{a'} \in \mathbb{D}_{m,n}$.

\begin{lemma}\label{lemma:hook_and_DRep}
Let $Y,Y'\in \mathcal{F}(Y_{m,n})$.
The following are equivalent.
\begin{enumerate}
    \item[(1)] there exists a box $(i,j)\in Y$ such that $Y' = Y\setminus h_Y{(i,j)}$.
    \item[(2)] There exist $-m<l\le r < n$ such that $\boldsymbol{d}(Y)\xrightarrow{l,r}\boldsymbol{d}(Y')$.
\end{enumerate}
In this case, it hold that $l = j-i'$ and $r = j'-i$, where $(i',j)$ is the bottom box in the $j$-th column of $Y$, and $(i,j')$ is the rightmost box in the $i$-th row of $Y$.
\end{lemma}

\begin{example}\label{ex:HR_on_DRep}
Let $Y$ be as in Example \ref{ex:DRep}, and let $Y' = Y\setminus h_{Y}(1,4)$.

\begin{center}
$Y=\ $\begin{ytableau}
    3&4&3&*(gray)2&*(gray)1\\
    2&3&4&*(gray)3\\
    1&2&3
\end{ytableau}
\qquad$\to$\qquad
$Y'=\ $\begin{ytableau}
    3&4&3\\
    2&3&4\\
    1&2&3
\end{ytableau}
\end{center}
In the diagonal expression, we see that 
\[
\boldsymbol{d}(Y) = (0,1,2,\dot{3},2,2,1,1,0),\qquad \boldsymbol{d}(Y') = (0,1,2,\dot{3},2,\underline{1,0,0},0),
\]
and $\boldsymbol{d}(Y)\xrightarrow{2,4}\boldsymbol{d}(Y')$.
\end{example}

\begin{proof}[Proof of Lemma \ref{lemma:hook_and_DRep}]
The implication (1)$\Rightarrow$(2) and equalities $l = j-i', r = j'-i$ follow from \eqref{eq:1}.
Let us show (2)$\Rightarrow$(1).
We give a proof only for the case that $l\le r \le 0$; the proof for the cases that $l \le 0 \le r$ and $0 \le l\le r$ are similar.
Notice that $d_l(Y),d_r(Y)>0$.
By Lemma \ref{lemma:diagonal} (2), we have $(d_l(Y)-l,d_l(Y)),(d_r(Y)-r,d_r(Y)) \in Y$.
Also, by the adjacency condition, it follows that $d_l(Y) \le d_r(Y)$.
Because $d_l(Y') = d_l(Y)-1$ and $d_r(Y') = d_r(Y)-1$, we see by Lemma \ref{lemma:diagonal} (2) that $(d_l(Y)-l,d_l(Y)),(d_r(Y)-r,d_r(Y)) \notin Y'$, which implied that $(d_l(Y)-l+1,d_l(Y)),(d_r(Y)-r,d_r(Y)+1) \notin Y'$.
Because $d_k(Y') = d_k(Y)$ for $k < l$ and $r < k$, we deduce that for $i,j$ such that $j-i < l$ or $r < j-i$, $(i,j)\in Y$ if and only if $(i,j)\in Y'$.
Hence, $(d_l(Y)-l+1,d_l(Y)),(d_r(Y)-r,d_r(Y)+1) \notin Y$.

Let $h$ be the hook in $Y$ corresponding to the box $(d_r(Y)-r,d_l(Y))$.
Because $(d_l(Y)-l,d_l(Y)) \in Y$ and $(d_l(Y)-l+1,d_l(Y)) \notin Y$, the bottom box in the $d_l(Y)$-th column of $Y$ is $(d_l(Y)-l,d_l(Y))$.
Also, because $(d_r(Y)-r,d_r(Y)) \in Y$ and $(d_r(Y)-r,d_r(Y)+1) \notin Y$, the rightmost box in the $(d_r(Y)-r)$-th row of $Y$ is $(d_r(Y)-r,d_r(Y))$.
We see from \eqref{eq:1} that the diagonal expression of $Y\setminus h$ is
\[
(\ldots,d_{l-1}(Y),d_{l}(Y)-1,d_{l+1}(Y)-1,\ldots,d_{r-1}(Y)-1,d_{r}(Y)-1,d_{r+1}(Y),\ldots),
\]
which is equal to $\boldsymbol{d}(Y')$.
Thus we have proved the lemma. \qedhere 
\end{proof}

The next lemma follows from the proof of Lemma \ref{lemma:hook_and_DRep}.

\begin{lemma}\label{lemma:hook_size}
Let $Y \in \mathcal{F}(Y_{m,n})$, and $Y' = Y\setminus h_Y{(i,j)}$ for $(i,j)\in Y$.
Let $-m < l \le r < n$ be such that $\boldsymbol{d}(Y)\xrightarrow{l,r}\boldsymbol{d}(Y')$ in the diagonal expression.
Then, $\#(h_Y{(i,j)}) = \#\mathcal{A}(h_Y{(i,j)}) = r-l+1$.
\end{lemma}

\begin{definition}\label{def:bulge}
Let $\boldsymbol{a} = (a_{-m},\ldots,a_{n}) \in \mathbb{N}_{0}^{m+n+1}$.
Assume that $(a_{k-1}, a_{k})$ satisfies the adjacency condition for some $-m<k\le n$.
If $(a_{k-1}-1, a_{k})$ (resp., $(a_{k-1}, a_{k}-1)$) also satisfies the adjacency condition, then we say that $(a_{k-1}, a_{k})$ is a left (resp., right) bulge, and we write $a_{k-1} \searrow a_{k}$ (resp., $a_{k-1} \nearrow a_{k}$).
\end{definition}

The following lemma can be easily verified.

\begin{lemma}\label{lemma:bulge}
Let $\boldsymbol{a} = (a_{-m},\ldots,a_{n}) \in \mathbb{N}_{0}^{m+n+1}$.
\begin{enumerate}
    \item[(1)] If $(a_{k-1}, a_{k})$ satisfies the adjacency condition, then $(a_{k-1}, a_{k})$ is either a left bulge or a right bulge.
    \item[(2)] Assume that $(a_{k-1}, a_{k})$ satisfies the adjacency condition.
    If $(a_{k-1}, a_{k})$ is a left bulge, then $(a_{k-1}-1, a_{k})$ is a right bulge.
    \item[(3)] Assume that $(a_{k-1}, a_{k})$ satisfies the adjacency condition.
    If $(a_{k-1}, a_{k})$ is a right bulge, then $(a_{k-1}, a_{k}-1)$ is a left bulge.
\end{enumerate}
\end{lemma}

\begin{lemma}\label{lemma:bulge_hook}
Let $\boldsymbol{a} = (a_{-m},\ldots, a_n), \boldsymbol{a'} = (a'_{-m},\ldots, a'_n) \in \mathbb{D}_{m,n}$.
Assume that $\boldsymbol{a'} = \boldsymbol{a}_{[l,r]} \in \mathbb{D}_{m,n}$ for some $-m < l\le r < n$.
Then, $a_{l-1}\nearrow a_{l}, a_{r}\searrow a_{r+1}$ and $a'_{l-1}\searrow a'_{l}, a'_{r}\nearrow a'_{r+1}$.
Moreover, for $-m<k\le n$ with $k\not=l,r+1$, if $a_{k-1}\nearrow a_{k}$ (resp., $a_{k-1}\searrow a_{k}$), then $a'_{k-1}\nearrow a'_{k}$ (resp., $a'_{k-1}\searrow a'_{k}$).
\end{lemma}
\begin{proof}
Since $\boldsymbol{a'} = \boldsymbol{a}_{[l,r]} \in \mathbb{D}_{m,n}$, it follows that $(a_{l-1}, a_{l}-1)$ and $(a_{r}-1, a_{r+1})$ satisfy the adjacency condition.
Hence, $(a_{l-1}, a_{l})$ is a right bulge, and $(a_{r}, a_{r+1})$ is a left bulge.
By Lemma \ref{lemma:bulge} $(2)$ and $(3)$, $(a'_{l-1}, a'_{l})$ is a left bulge, and $(a'_{r}, a'_{r+1})$ is a right bulge.
By the definition of $\boldsymbol{a}_{[l,r]}$, we have $a_{k} - a_{k-1} = a'_{k} - a'_{k-1}$ for $-m<k\le n$ with $k\not=l,r+1$.
Hence, $(a_{l-1}, a_{l})$ and $(a'_{l-1}, a'_{l})$ are both left bulges or rights bulge.
Thus we have proved the lemma.\qedhere 
\end{proof}

Let $Y\in\mathcal{F}(Y_{m,n})$ be a Young diagram with the unimodal numbering $\alpha_{m,n}$.
By Remark \ref{remark:unimodal} (2), it holds that $\alpha_{m,n}(i',j') = \alpha_{m,n}(i'+a,j'+a)$ for all $(i',j')\in Y$ and $a\in\mathbb{N}$ such that $(i'+a,j'+a)\in Y$.
Hence we see that $\mathcal{A}(Y) = \mathcal{A}(Y\setminus h_Y{(i,j)}) \cup \mathcal{A}(h_Y{(i,j)})$ for $(i,j)\in Y$.

\begin{lemma}\label{lemma:DRep_C}
For $Y\in\mathcal{F}(Y_{m,n})$ and $1\le k \le \hat{\alpha}_{m,n} = \lfloor (n+m)/2 \rfloor$,
\[
\#\{x\in \mathcal{A}(Y)\mid x=k\} =
\begin{cases}
    d_{-m+k}+d_{n-k} & \text{if}\ -m+k\not=n-k,\\
    d_{-m+k} & \text{if}\ -m+k=n-k.
\end{cases}
\]
\end{lemma}
\begin{proof}
Assume that $-m+k\not=n-k$.
Then we compute
\begin{align*}
&\#\{x\in \mathcal{A}(Y)\mid x=k\} = \#\{(i,j)\in Y\mid j-i = -m+k\ \text{or}\ n-k\} \\
=\ & \#\{(i,j)\in Y\mid j-i = -m+k\} +\#\{(i,j)\in Y\mid j-i = n-k\} \\
=\ & d_{-m+k}+d_{n-k}.
\end{align*}
The proof of the case that $-m+k=n-k$ is similar.
Thus we have proved the lemma.\qedhere 

\end{proof}

\begin{lemma}\label{lemma:same_number_hook}
Let $Y \in \mathcal{F}(Y_{m,n})$, and $Y' = Y\setminus h_Y{(i,j)}$ for $(i,j)\in Y$.
Let $-m < l \le r < n$ be such that $\boldsymbol{d}(Y)\xrightarrow{l,r}\boldsymbol{d}(Y')$ in the diagonal expression (see \eqref{eq:1}).
Assume that there exists $(i',j') \in Y'$ such that $\mathcal{A}(h_{Y'}(i',j')) = \mathcal{A}(h_Y{(i,j)})$.
Set $Y'' = Y'\setminus h_{Y'}(i',j')$.
Then, $\boldsymbol{d}(Y')\xrightarrow{n-m-r,n-m-l}\boldsymbol{d}(Y'')$ in the diagonal expression.
Also, there exists no box $(i'',j'') \in Y''$ such that $\mathcal{A}(h_{Y''}(i'',j'')) = \mathcal{A}(h_Y{(i,j)})$.
\end{lemma}

\begin{example}\label{ex:same_number_hook}
Let $Y$ be as in Example \ref{ex:DRep} (note that $m = 3$ and $n = 5$), and set $Y' = Y\setminus h_{Y}(1,3)$.
Then we have $\mathcal{A}(h_{Y}(1,3))=\{\!\{3,4,3,2,1\}\!\}$.
Notice that $\mathcal{A}(h_{Y'}(1,1))=\{\!\{1,2,3,4,3\}\!\}=\mathcal{A}(h_{Y}(1,3))$.
We set $Y'' = Y'\setminus h_{Y'}(1,1)$.
\begin{center}
$Y=\ $\begin{ytableau}
    3&4&*(gray)3&*(gray)2&*(gray)1\\
    2&3&*(gray)4&3\\
    1&2&*(gray)3
\end{ytableau}
\qquad$\to$\qquad
$Y'=\ $\begin{ytableau}
    *(gray)3&*(gray)4&*(gray)3\\
    *(gray)2&3\\
    *(gray)1&2
\end{ytableau}
\qquad$\to$\qquad
$Y''=\ $\begin{ytableau}
    3\\
    2
\end{ytableau}
\end{center}
In this case, 
\[\boldsymbol{d}(Y) = (0,1,2,\dot{3},2,2,1,1,0),\]
\[\boldsymbol{d}(Y') = (0,1,2,\underline{\dot{2},1,1,0,0},0),\]
\[\boldsymbol{d}(Y'') = (0,\underline{0,1,\dot{1},0,0},0,0,0),\]
and hence $\boldsymbol{d}(Y)\xrightarrow{0,4}\boldsymbol{d}(Y')\xrightarrow{-2,2}\boldsymbol{d}(Y'')$, where $-2 = 5-3-4$ and $2 = 5-3-0$.
\end{example}

\begin{proof}[Proof of Lemma \ref{lemma:same_number_hook}]
We set $h:= h_Y{(i,j)}$ and $h':= h_{Y'}(i',j')$ for simplicity of notation.
Because $Y'' = Y'\setminus h'$, we see by \eqref{eq:1} that $\boldsymbol{d}(Y') \xrightarrow{l',r'} \boldsymbol{d}(Y'')$ for some $-m < l' \le r' < n$; we will show that $l' = m-n-r$ and $r' = n-m-l$.
Since $\mathcal{A}(h')=\mathcal{A}(h)$, and since $\boldsymbol{d}(Y)\xrightarrow{l,r}\boldsymbol{d}(Y')$, we have $\#\mathcal{A}(h') = \#\mathcal{A}(h) = r-l+1$ by Lemma \ref{lemma:hook_size}.
Hence we see that $\boldsymbol{d}(Y')\xrightarrow{a,a+r-l}\boldsymbol{d}(Y'')$ for some $a\in \mathbb{Z}$;
it suffices to show that $a = n-m-r$.

First, suppose, for a contradiction, that $a = l$; 
note that $\boldsymbol{d}(Y')\xrightarrow{l,r}\boldsymbol{d}(Y'')$ in this case.
Hence we see by Lemma \ref{lemma:bulge_hook} that $d_{l-1}(Y') \nearrow d_{l}(Y')$.
Similarly, since $\boldsymbol{d}(Y)\xrightarrow{l,r}\boldsymbol{d}(Y')$, it follows from Lemma \ref{lemma:bulge_hook} that $d_{l-1}(Y') \searrow d_{l}(Y')$.
Thus we get $d_{l-1}(Y') \nearrow d_{l}(Y')$ and $d_{l-1}(Y') \searrow d_{l}(Y')$, which contradicts Lemma \ref{lemma:bulge} (1).

Next, suppose, for a contradiction, that $a \not = l, n-m-r$.
For $k\in\mathbb{Z}$, We define $\mu(k) := \min\{k+m,-k+n\}$.
Since $\boldsymbol{d}(Y')\xrightarrow{a,a+r-l}\boldsymbol{d}(Y'')$, we have $\mathcal{A}(h')=\{\!\{\ \alpha_{m,n}(i',j')\mid (i',j')\in h'\}\!\}=\{\!\{\ \min\{j'-i'+m,i'-j'+n\}\mid (i',j')\in h'\}\!\}=\{\!\{\ \mu(k)\mid a\le k\le a+r-l\}\!\}$.
Note that $\min \mathcal{A}(h)=\min \{\!\{\ \min\{j'-i'+m,i'-j'+n\}\mid (i',j')\in h\}\!\} = \min\{\min\{l+m,l+n\},\min\{r+m,r+n\}\} = \min\{\mu(l),\mu(r)\}$. 
We give a proof only for the case that that $\mu(l) < \mu(r)$; the proofs for the cases that $\mu(l) = \mu (r)$ and $\mu(l) > \mu (r)$ are similar. 
If $l \ge n-m-l$, then 
\begin{align*}
    \mu(l) &= \min\{l+m,-l+n\} = m+\min\{l,-l+n-m\} = n-l \\
    &\ge \min\{r+m,(n-l)+\underbrace{(l-r)}_{\le 0}\} = \min\{r+m,-r+n\} = \mu(r),
\end{align*}
which is a contradiction.
Hence we get $l < n-m-l$ and $\mu(l)=\mu(n-m-l) = l+m$.
If $l < a < n-m-r$, then $a+r-l < n-m-l$.
Then, we have $\mu(b) = \min\{b+m,-b+n\} = \min\{a+m,-a-r+l+n\} > \min\{l+m,-n+m+l+n\} = l+m = \mu(l)$ for $a \le b \le a+r-l$.
Since $\mathcal{A}(h') =\{\!\{\mu(k)\mid a\le k \le a+r-l\}\!\}$, it follows that $\mu(l)\in \mathcal{A}(h)$ is not contained in $\mathcal{A}(h')$, which is a contradiction.
If $a < l$, then $a+m < l+m < n-m-l+m < -a+n$ and $\mu(a) = \min\{a+m,-a+n\} = a+m < l+m = \mu(l) = \min\mathcal{A}(h)$.
hence we get $\mu(a)\notin\mathcal{A}(h)$, which is a contradiction.
If $n-m-r < a$, then $a+r-l+m > -l+n > l+m > -a-r+l+n$ and $\mu(a+r-l) = \min\{a+r-l+m,-a-r+l+n\} = -a-r+l+n < l+m = \mu(l) = \min\mathcal{A}(h)$.
hence we get $\mu(a+r-l)\notin\mathcal{A}(h)$, which is a contradiction.
Thus we obtain $a = n-m-r$, as desired.

Suppose, for a contradiction, that there exists a box $(i'',j'') \in Y''$ such that $\mathcal{A}(h'') = \mathcal{A}(h)$, where $h'' := h_{Y''}(i'',j'')$; note that $\mathcal{A}(h'') = \mathcal{A}(h')$.
Since $\boldsymbol{d}(Y')\xrightarrow{n-m-r,n-m-l}\boldsymbol{d}(Y'')$, it follows by the argument above that $\boldsymbol{d}(Y''\setminus h'')$ is equal to $\boldsymbol{d}(Y'')_{[l,r]}$ or $\boldsymbol{d}(Y'')_{[n-m-r,n-m-l]}$.

If $\boldsymbol{d}(Y''\setminus h'') = \boldsymbol{d}(Y'')_{[l,r]}$, then we see by Lemma \ref{lemma:bulge_hook} that $d_{l-1}(Y'') \nearrow d_{l}(Y'')$ and $d_{r}(Y'') \searrow d_{r+1}(Y'')$.
Similarly, since $\boldsymbol{d}(Y)\xrightarrow{l,r}\boldsymbol{d}(Y')$, it follows from Lemma \ref{lemma:bulge_hook} that $d_{l-1}(Y') \searrow d_{l}(Y')$ and $d_{r}(Y') \nearrow d_{r+1}(Y')$.
Note that $\boldsymbol{d}(Y')\xrightarrow{n-m-r,n-m-l}\boldsymbol{d}(Y'')$.
If $l = n-m-l+1$, then $r \ge l = n-m-l+1 \ge n-m-r+1 > n-m-r-1$.
Thus we see by Lemma \ref{lemma:bulge_hook} that $d_{l-1}(Y'') \searrow d_{l}(Y'')$ or $d_{r}(Y'') \nearrow d_{r+1}(Y'')$.
Thus we get $d_{l-1}(Y'') \nearrow d_{l}(Y'')$ and $d_{l-1}(Y'') \searrow d_{l}(Y'')$, or $d_{r}(Y'') \searrow d_{r+1}(Y'')$ and $d_{r}(Y'') \nearrow d_{r+1}(Y'')$, which contradicts Lemma \ref{lemma:bulge} (1).

If $\boldsymbol{d}(Y''\setminus h'') = \boldsymbol{d}(Y'')_{[n-m-r,n-m-l]}$, then we see by Lemma \ref{lemma:bulge_hook} that $d_{n-m-r-1}(Y'') \nearrow d_{n-m-r}(Y'')$.
Similarly, since $\boldsymbol{d}(Y')\xrightarrow{n-m-r,n-m-l}\boldsymbol{d}(Y'')$, it follows from Lemma \ref{lemma:bulge_hook} that $d_{n-m-r-1}(Y'') \searrow d_{n-m-r}(Y'')$.
Thus we get $d_{n-m-r-1}(Y'') \nearrow d_{n-m-r}(Y'')$ and $d_{n-m-r-1}(Y'') \searrow d_{n-m-r}(Y'')$, which contradicts Lemma \ref{lemma:bulge} (1).\qedhere 
\end{proof}


\section{$\mathcal{G}$-values of Multiple Hook Removing Game whose starting position is $Y_{m,n}$}

\subsection{Multiple Hook Removing Game whose starting position is $Y_{m,n}$}

Let and fix $m,n\in\mathbb{N}$.
Recall that the boxes in $Y\in\mathcal{F}(Y_{m,n})$ are numbered by the unimodal numbering $\alpha_{m,n}:Y\to \mathbb{N}$, where $\alpha_{m,n}(i,j)=\min\{j-i+m, i-j+n\} \in \mathbb{N}$ for $(i,j) \in Y$.
We denote MHRG (see Section 1.1) whose starting position is the rectangular Young diagram $Y_{m,n}$ (with the unimodal numbering $\alpha_{m,n}$) by MHRG$(m,n)$;
it can be easily shown that MHRG$(m,n)$ is isomorphic to MHRG$(n,m)$;
in what follow, we assume that $m\le n$.
Let $\mathcal{T}(Y_{m,n})$ be the subset of $\mathcal{F}(Y_{m,n})$ consisting of all $Y\in\mathcal{F}(Y_{m,n})$ such that there exists a transition from $Y_{m,n}$ to $Y$, that is, $\mathcal{T}(Y_{m,n}) = \mathcal{C}(\text{MHRG}(m,n))$.

\begin{remark}[see Theorem \ref{th:1n} and Lemma \ref{lemma:2n_transition} below; see also \cite{Motegi}]\label{re:t_not=_f}
In general, it does not hold that $\mathcal{T}(Y_{m,n}) = \mathcal{F}(Y_{m,n})$. 
\end{remark}

\begin{remark}
We see by Lemma \ref{lemma:same_number_hook} that in MHRG$(m,n)$, the operation (M4b) is performed at most once, and the operation (M4c) is never performed.
\end{remark}

Let $Y\in \mathcal{T}(Y_{m,n})$, and $Y'\in \mathcal{O}(Y)$.
By Lemmas \ref{lemma:hook_and_DRep} and \ref{lemma:same_number_hook}, there exists $-m<l\le r<n$ such that 
\[
\boldsymbol{d}(Y)\xrightarrow{l,r}\boldsymbol{d}(Y'),
\]
or there exist $-m<l\le r<n$ and $Y''\in \mathcal{F}(Y_{m,n})$ such that  
\[
\boldsymbol{d}(Y)\xrightarrow{l,r}\boldsymbol{d}(Y'')\xrightarrow{n-m-r,n-m-l}\boldsymbol{d}(Y').
\]

It can be easily shown that MHRG$(m,n)$ is isomorphic to MHRG$(n,m)$;
in what follow, we assume that $m\le n$.


\subsection{Relation between MHRG$(m,n)$ and MHRG$(m,n+1)$}

Fix $m,n\in\mathbb{N}$ with $m\le n$.
Assume that $m+n$ is even.
We define $c := (n-m)/2$;
note that $c$ is a non-negative integer.
We will prove that MHRG$(m,n)$ is isomorphic to MHRG$(m,n+1)$.

\begin{definition}\label{def:index_cor}
We define the map $E:\mathbb{N}_0^{m+n+1} \to \mathbb{N}_0^{m+n+2}$ as follows.
If $\boldsymbol{a} \in \mathbb{N}_0^{m+n+1}$ is
\[
\boldsymbol{a} = (a_{-m}, \ldots , a_{c-1}, \underbrace{a_{c}}_{c\text{-th}}, a_{c+1},  \ldots ,a_{n}),
\]
then
\[
E(\boldsymbol{a}) := (a_{-m}, \ldots , a_{c-1}, \underbrace{a_{c}}_{c\text{-th}}, \underbrace{a_{c}}_{(c+1)\text{-th}}, a_{c+1},  \ldots ,a_{n}).
\]
\end{definition}

It can be easily verified that 
\begin{equation}\label{eq:aEa}
\boldsymbol{a} \in \mathbb{D}_{m,n} \text{ if and only if } E(\boldsymbol{a}) \in \mathbb{D}_{m,n+1}. 
\end{equation}
Hence the map $E:\mathbb{N}_0^{m+n+1} \to \mathbb{N}_0^{m+n+2}$ induces the map $E:\mathcal{F}(Y_{m,n}) \to \mathcal{F}(Y_{m,n+1})$ as follows.
For $Y \in \mathcal{F}(Y_{m,n})$, we define $E(Y)$ to the (unique) element of ${F}(Y_{m,n+1})$ whose diagonal expression is
\[
E(\boldsymbol{d}(Y)) = (d_{-m}(Y), \ldots , d_{c-1}(Y), d_{c}(Y), d_{c}(Y), d_{c+1}(Y),  \ldots ,d_{n-1}(Y));
\]
note that $\boldsymbol{d}(E(Y)) = E(\boldsymbol{d}(Y))$.
Notice that $E:\mathcal{F}(Y_{m,n}) \to \mathcal{F}(Y_{m,n+1})$ is an injection.

\begin{theorem}\label{th:game_cor}
Let $m,n\in\mathbb{N}$ be such that $m\le n$ and $m+n$ is even.
Then the map $E$ gives an isomorphism from MHRG$(m,n)$ to MHRG$(m,n+1)$.
Therefore, for each $Y\in \mathcal{T}(Y_{m,n})$, it holds that $\mathcal{G}(Y) = \mathcal{G}(E(Y))$. In particular, $\mathcal{G}(Y_{m,n})$ in MHRG$(m,n)$ is equal to $\mathcal{G}(Y_{m,n+1})$ in MHRG$(m,n+1)$.
\end{theorem}

In order to prove Theorem \ref{th:game_cor}, we need some lemmas.
For $l,r\in \mathbb{Z}$, we define $e_l,e_r : \mathbb{Z} \to \mathbb{Z}$ by
\begin{equation*}
    e_l(k) := \left\{
    \begin{aligned}
        &k\ &\text{if}\ k \le c,\\
        &k+1\ &\text{if}\ k > c,
    \end{aligned}\right.\qquad
    e_r(k) := \left\{
    \begin{aligned}
        &k\ &\text{if}\ k < c,\\
        &k+1\ &\text{if}\ k \ge c.
    \end{aligned}\right.
\end{equation*}
Note that $e_l(k) \not= c+1$ and $e_r(k) \not= c$.
The following lemma can be shown easily.

\begin{lemma}\label{lemma:e_lr}
Let $l,r\in \mathbb{Z}$.
It holds that $e_l(n-m-k) = n-m+1-e_r(k)$ for $k\in\mathbb{Z}$.
\end{lemma}

\begin{lemma}\label{lemma:hook_cor}
Let $l,r\in\mathbb{Z}$.
For $\boldsymbol{a} \in \mathbb{N}_0^{m+n+1}$, it holds that $E(\boldsymbol{a}_{[l,r]})=E(\boldsymbol{a})_{[e_l(l),e_r(r)]}$.
Therefore, for $Y\in\mathcal{F}(Y_{m,n})$, it holds that $\boldsymbol{d}(Y)_{[l,r]} \in \mathbb{D}_{m,n}$ if and only if $\boldsymbol{d}(E(Y))_{[e_l(l),e_r(r)]}  \in \mathbb{D}_{m,n+1} $.
\end{lemma}
\begin{proof}
If $c < l \le r$, then $l+1=e_l(l), r+1= e_r(r)$, and
\[
E(\boldsymbol{a}_{[l,r]}) = (\ldots , \underbrace{a_{c}}_{c\text{-th}}, \underbrace{a_{c}}_{(c+1)\text{-th}}, \ldots , a_{l-1}, \underbrace{a_{l}-1}_{(l+1)\text{-th}}, \ldots , \underbrace{a_{r}-1}_{(r+1)\text{-th}}, a_{r+1} , \ldots).
\]
Thus we get $E(\boldsymbol{a}_{[l,r]})= E(\boldsymbol{a})_{[l+1,r+1]} = E(\boldsymbol{a})_{[e_l(l),e_r(r)]}$.

If $l \le c \le r$, then $l=e_l(l), r+1= e_r(r)$, and
\[
E(\boldsymbol{a}_{[l,r]}) = (\ldots , a_{l-1}, \underbrace{a_{l}-1}_{l\text{-th}}, \ldots , \underbrace{a_{c}-1}_{c\text{-th}}, \underbrace{a_{c}-1}_{(c+1)\text{-th}}, \ldots , \underbrace{a_{r}-1}_{(r+1)\text{-th}}, a_{r+1} , \ldots).
\]
Thus we get $E(\boldsymbol{a}_{[l,r]})= E(\boldsymbol{a})_{[l,r+1]} = E(\boldsymbol{a})_{[e_l(l),e_r(r)]}$.

If $ l \le r < c $, then $l=e_l(l), r= e_r(r)$, and
\[
E(\boldsymbol{a}_{[l,r]}) = (\ldots , a_{l-1}, \underbrace{a_{l}-1}_{l\text{-th}}, \ldots , \underbrace{a_{r}-1}_{r\text{-th}}, a_{r+1} , \ldots, \underbrace{a_{c}}_{c\text{-th}}, \underbrace{a_{c}}_{(c+1)\text{-th}}, \ldots ).
\]
Thus we get $E(\boldsymbol{a}_{[l,r]})= E(\boldsymbol{a})_{[l,r]} = E(\boldsymbol{a})_{[e_l(l),e_r(r)]}$.

In all cases above, we have $E(\boldsymbol{a}_{[l,r]}) = E(\boldsymbol{a})_{[e_l(l),e_r(r)]}$ for $-m < l \le r < n$,. 
Hence, by $\boldsymbol{d}(E(Y)) = E(\boldsymbol{d}(Y))$ and \eqref{eq:aEa}, we obtain 
\begin{align*}
\boldsymbol{d}(Y)_{[l,r]} \in \mathbb{D}_{m,n} &\overset{\eqref{eq:aEa}}{\Leftrightarrow} E(\boldsymbol{d}(Y)_{[l,r]}) \in \mathbb{D}_{m,n+1}\\
&\Leftrightarrow E(\boldsymbol{d}(Y))_{[e_l(l),e_r(r)]} \in \mathbb{D}_{m,n+1}\\
&\Leftrightarrow \boldsymbol{d}(E(Y))_{[e_l(l),e_r(r)]} \in \mathbb{D}_{m,n+1},
\end{align*}
as desired.
\end{proof}

\begin{lemma}\label{lemma:step_cor}
Let $Y, Y'\in \mathcal{T}(Y_{m,n})$.
Assume that $Y \rightarrow Y'$, and $E(Y) \in \mathcal{T}(Y_{m,n+1})$.
Then, $E(Y') \in \mathcal{T}(Y_{m,n+1})$, and $E(Y) \rightarrow E(Y')$.
\end{lemma}
\begin{proof}
Because $Y \rightarrow Y'$, we see by the definition that 
\begin{enumerate}
    \item[(a)] there exists $-m<l\le r<n$ such that $\boldsymbol{d}(Y)\xrightarrow{l,r}\boldsymbol{d}(Y')$, or
    \item[(b)] there exist $-m<l\le r<n$ and $Y''\in \mathcal{F}(Y_{m,n})$ such that $\boldsymbol{d}(Y)\xrightarrow{l,r}\boldsymbol{d}(Y'')\xrightarrow{n-m-r,n-m-l}\boldsymbol{d}(Y')$.
\end{enumerate}
We give a proof only for the case (b);
the proof for the case (a) is similar and simpler.

By Lemma \ref{lemma:e_lr}, we have $e_l(n-m-r) = n-m+1-e_r(r)$ and $e_r(n-m-l) = n-m+1-e_l(l)$.
Thus we have 
\[
\boldsymbol{d}(E(Y))\xrightarrow{e_l(l),e_r(r)}\boldsymbol{d}(E(Y''))\xrightarrow{e_l(n-m-r),e_r(n-m-l)}\boldsymbol{d}(E(Y'))
\]
by Lemma \ref{lemma:hook_cor}, which implies that $E(Y) \rightarrow E(Y')$.
Thus we have proved the lemma.
\end{proof}

Let $Y' \in \mathcal{T}(Y_{m,n})$, and let $Y_{m,n} = Y_0 \xrightarrow{}Y_1\xrightarrow{}\cdots\xrightarrow{} Y_k = Y'$ be a transition from $Y_{m,n}$ to $Y'$ in MHRG$(m,n)$.
Note that $E(Y_0) = E(Y_{m,n}) = Y_{m,n+1} \in \mathcal{T}(Y_{m,n+1})$.
Also, we see by Lemma \ref{lemma:step_cor} that for $0 \le p < k$, if $E(Y_p) \in \mathcal{T}(Y_{m,n+1})$, then $E(Y_{p+1}) \in \mathcal{T}(Y_{m,n+1})$.
Thus we obtain $E(Y') \in \mathcal{T}(Y_{m,n+1})$ by inductive argument.
Therefore, we obtain 
\begin{equation}\label{eq:EC}
E(\mathcal{T}(Y_{m,n})) \subset \mathcal{T}(Y_{m,n+1}).
\end{equation}
Moreover, it is obvious by Lemma \ref{lemma:step_cor} that 
\begin{equation}\label{eq:ENCNE}
E(\mathcal{O}(Y)) \subseteq \mathcal{O}(E(Y))
\end{equation}
for $Y \in \mathcal{T}(Y_{m,n+1})$.

\begin{lemma}\label{lemma:same_center}
It holds that $d_{c}(Y) = d_{c+1}(Y)$ for all $Y \in \mathcal{T}(Y_{m,n+1})$.
\end{lemma}
\begin{proof}
Suppose, for a contradiction, that there exists $Y \in \mathcal{T}(Y_{m,n+1})$ such that $d_{c}(Y) \not= d_{c+1}(Y)$.
Let $\mathcal{V} \subset \mathcal{T}(Y_{m,n+1})$ be the subset of $\mathcal{T}(Y_{m,n+1})$ consisting of elements $Y \in \mathcal{T}(Y_{m,n+1})$ such that $d_{c}(Y) \not= d_{c+1}(Y)$, and let $Y_0 \in \mathcal{V}$ be such that $\#(Y_0) \ge \#(Y)$ for all $T \in \mathcal{V}$.
Because $c \ge 0$, and $(d_{c}(Y_0), d_{c+1}(Y_0))$ satisfies the adjacency condition, we have $d_{c}(Y_0) = d_{c+1}(Y_0) +1$ and $d_{c}(Y_0)\searrow d_{c+1}(Y_0)$.

Because $Y_0 \not= Y_{m,n+1}$, there exists $Y_1 \in \mathcal{T}(Y_{m,n+1})$ such that $Y_1 \rightarrow Y_0$;
note that $\#(Y_1) \ge \#(Y_0)$, which implies that $Y_1 \notin \mathcal{V}$ by the maximality of $Y_0$.
Thus we have $d_{c}(Y_1) = d_{c+1}(Y_1)$, and $d_{c}(Y_1)\nearrow d_{c+1}(Y_1)$.
By Lemma \ref{lemma:DRep_C}, we set that for $p=0,1$, the number $t_p$ of boxes in $Y_p$ having the number $\hat{\alpha}_{m,n} = (m+n)/2$ is equal to $d_{c}(Y_p) + d_{c+1}(Y_p)$.
Thus, $t_1-t_0$ is odd.
If two hooks are removed in $Y_1 \rightarrow Y_0$, then the two hooks have a same multiset of numbers.
Thus, $t_1-t_0$ is even, which is a contradiction.
Thus, one hook is removed in $Y_1 \rightarrow Y_0$, and hence $d_{c}(Y_0) = d_{c}(Y_1), d_{c+1}(Y_0)= d_{c+1}(Y_1)-1$ by $0 \le d_{k}(Y_1) - d_{k}(Y_0) \le 1$ for $-m \le k \le n$.
Also, there exists $c+1 \le k = k(Y_1) < n+1$ such that $\boldsymbol{d}(Y_1)\xrightarrow{c+1,k}\boldsymbol{d}(Y_0)$, and $\boldsymbol{d}(Y_0)_{[n+1-m-k,c]} \notin \mathbb{D}_{m,n+1}$; note that $n+1-m-(c+1)=c$.
By Lemma \ref{lemma:bulge_hook}, we have $d_{n-m-k}(Y_1)\searrow d_{n+1-m-k}(Y_1), d_{c}(Y_1)\nearrow d_{c+1}(Y_1), d_{k}(Y_1)\searrow d_{k+1}(Y_1)$.
Now, we choose $Y_1$ such that $k = k(Y_1)$ is maximum.

Suppose that $Y_1 = Y_{m,n+1}$.
In this case, we have $d_{p}(Y_1)\nearrow d_{p+1}(Y_1)$ for $-m \le p < n-m $, and $d_{p}(Y_1)\searrow d_{p+1}(Y_1)$ for $n-m \le p \le n$.
Because $c \le n-m$, we have $d_{n-m-k}(Y_1)\nearrow d_{n+1-m-k}(Y_1)$, and $d_{n-m-k}(Y_0)\nearrow d_{n+1-m-k}(Y_0)$ by Lemma \ref{lemma:bulge_hook}.
Thus we have $\boldsymbol{d}(Y_0)_{[n+1-m-k,c]}\in \mathbb{D}_{m,n+1}$ by Lemma \ref{lemma:bulge_hook}, which is a contradiction.
Hence we obtain $Y_1 \not= Y_{m,n+1}$.
Then, there exists $Y_2 \in \mathcal{T}(Y_{m,n+1})$ such that $Y_2 \rightarrow Y_1$.
Note that $d_{c}(Y_2)  = d_{c+1}(Y_2) $ and $d_{c}(Y_2) \nearrow d_{c+1}(Y_2) $.

Suppose that $d_{n-m-k}(Y_2) \searrow d_{n+1-m-k}(Y_2) $ and $d_{k}(Y_2) \searrow d_{k+1}(Y_2) $.
By Lemma \ref{lemma:bulge_hook}, we have $\boldsymbol{d}(Y_2)_{[c+1,k]} \in \mathbb{D}_{m,n+1}$.
Let $Y_1' \in \mathcal{F}(Y_{m,n+1})$ be the Young diagram whose diagonal expression is equal to $\boldsymbol{d}(Y_2)_{[c+1,k]}$;
notice that $d_{c}(Y_1') \not= d_{c+1}(Y_1')$ and $d_{n-m-k}(Y_1') \searrow d_{n+1-m-k}(Y_1') $.
Since $\boldsymbol{d}(Y_1')_{[n+1-m-k,c]} \notin \mathbb{D}_{m,n+1}$, it follows that $Y_1' \in \mathcal{O}(Y_2)$, and hence $Y_1'\in \mathcal{V}$.
Because $\#(Y_1)-\#(Y_0)=\#(Y_2)-\#(Y_1')$, we have $\#(Y_1') = \#(Y_2)-\#(Y_1)+\#(Y_0) > \#(Y_0)$, which contradicts the maximality of $Y_0$.

Suppose that $d_{n-m-k}(Y_2) \searrow d_{n+1-m-k}(Y_2) $ and $d_{k}(Y_2) \nearrow d_{k+1}(Y_2) $.
If two hooks are removed in $Y_2 \rightarrow Y_1$, then there exist $-m<l\le r<n$ and $Y'\in \mathcal{F}(Y_{m,n})$ such that
\[
\boldsymbol{d}(Y_2)\xrightarrow{l,r}\boldsymbol{d}(Y')\xrightarrow{n+1-m-r,n+1-m-l}\boldsymbol{d}(Y_1).
\]
Because $d_{k}(Y_2) \nearrow d_{k+1}(Y_2) $ and $d_{k}(Y_1) \searrow d_{k+1}(Y_1) $, we have 
\[
\boldsymbol{d}(Y_2)\xrightarrow{k+1,r}\boldsymbol{d}(Y')\xrightarrow{n+1-m-r,n-m-k}\boldsymbol{d}(Y_1)
\]
or 
\[
\boldsymbol{d}(Y_2)\xrightarrow{l,n-m-k}\boldsymbol{d}(Y')\xrightarrow{k+1,n+1-m-l}\boldsymbol{d}(Y_1).
\]
Thus we have $d_{n-m-k}(Y_1) \nearrow d_{n+1-m-k}(Y_1) $, which is a contradiction.
Hence one hook is removed in $Y_2 \rightarrow Y_1$.
Then there exist $p \ge k+1$ such that
\[
\boldsymbol{d}(Y_2)\xrightarrow{k+1, p}\boldsymbol{d}(Y_1).
\]
Note that $\boldsymbol{d}(Y_1)_{[n+1-m-p,n-m-k]} \notin \mathbb{D}_{m,n+1}$ and $d_{n-m-p}(Y_2) \searrow d_{n+1-m-p}(Y_2)$, $d_{p}(Y_2) \searrow d_{p+1}(Y_2)$.
By Lemma \ref{lemma:bulge_hook}, we have $\boldsymbol{d}(Y_2)_{[c+1,p]} \in \mathbb{D}_{m,n+1}$.
Let $Y_1' \in \mathcal{F}(Y_{m,n+1})$ be the Young diagram whose diagonal expression is equal to $\boldsymbol{d}(Y_2)_{[c+1,p]}$;
notice that $d_{c}(Y_1') \not= d_{c+1}(Y_1')$.
Since $\boldsymbol{d}(Y_1')_{[n+1-m-p,c]} \notin \mathbb{D}_{m,n+1}$, it follows that $Y_1' \in \mathcal{O}(Y_2)$, and hence $Y_1'= \boldsymbol{d}(Y_2)_{[c+1,p]} = (\boldsymbol{d}(Y_2)_{[k+1, p]})_{[c+1,k]} = Y_0$, which contradicts the maximality of $k$.

Suppose that $d_{n-m-k}(Y_2) \nearrow d_{n+1-m-k}(Y_2) $ and $d_{k}(Y_2) \searrow d_{k+1}(Y_2) $.
If two hooks are removed in $Y_2 \rightarrow Y_1$, then there exist $-m<l\le r<n$ and $Y'\in \mathcal{F}(Y_{m,n})$ such that
\[
\boldsymbol{d}(Y_2)\xrightarrow{l,r}\boldsymbol{d}(Y')\xrightarrow{n+1-m-r,n+1-m-l}\boldsymbol{d}(Y_1).
\]
Because $d_{n-m-k}(Y_2) \nearrow d_{n+1-m-k}(Y_2) $ and $d_{n-m-k}(Y_1) \searrow d_{n+1-m-k}(Y_1) $, we have 
\[
\boldsymbol{d}(Y_2)\xrightarrow{n+1-m-k,r}\boldsymbol{d}(Y')\xrightarrow{n+1-m-r,k}\boldsymbol{d}(Y_1)
\]
or
\[
\boldsymbol{d}(Y_2)\xrightarrow{l,k}\boldsymbol{d}(Y')\xrightarrow{n+1-m-k,n+1-m-l}\boldsymbol{d}(Y_1).
\]
Thus we have $d_{k}(Y_1) \nearrow d_{k+1}(Y_1) $, which is a contradiction.
Hence one hook is removed in $Y_2 \rightarrow Y_1$.
Then there exist $p \ge n+1-m-k$ such that
\[
\boldsymbol{d}(Y_2)\xrightarrow{n+1-m-k,p}\boldsymbol{d}(Y_1).
\]
Note that $\boldsymbol{d}(Y_1)_{[n+1-m-p,k]} \notin \mathbb{D}_{m,n+1}$ and $d_{n-m-p}(Y_2) \searrow d_{n+1-m-p}(Y_2)$, $d_{p}(Y_2) \searrow d_{p+1}(Y_2)$.
Because $d_{c}(Y_2) \nearrow d_{c+1}(Y_2)$, we have $p \not= c$. 
By Lemma \ref{lemma:bulge_hook}, we have $\boldsymbol{d}(Y_2)_{[c+1,\max(p,n+1-m-p)]} \in \mathbb{D}_{m,n+1}$; notice that $c+1 \le \max(p,n+1-m-p)$ since $p \not= c$.
Let $Y_1' \in \mathcal{F}(Y_{m,n+1})$ be the Young diagram whose diagonal expression is equal to $\boldsymbol{d}(Y_2)_{[c+1,\max(p,n+1-m-p)]}$;
note that $d_{c}(Y_1') \not= d_{c+1}(Y_1')$.
Since $\boldsymbol{d}(Y_1')_{[\min\{p,n+1-m-p\},c]} \notin \mathbb{D}_{m,n+1}$, it follows that $Y_1' \in \mathcal{O}(Y_2)$, and hence $Y_1'\in \mathcal{V}$.
If $\max(p,n+1-m-p) \le k$, then $\#(Y_1)-\#(Y_0) \ge \#(Y_2)-\#(Y_1')$, and hence $\#(Y_1') \ge \#(Y_2)-\#(Y_1)+\#(Y_0) > \#(Y_0)$.
If $\max(p,n+1-m-p) > k$, then $\#(Y_2)-\#(Y_1) > \#(Y_2)-\#(Y_1')$, and hence $\#(Y_1') > \#(Y_1) > \#(Y_0)$.
In any case, we obtain $\#(Y_1') > \#(Y_0)$, which contradicts the maximality of $Y_0$.

Suppose that $d_{n-m-k}(Y_2) \nearrow d_{n+1-m-k}(Y_2) $ and $d_{k}(Y_2) \nearrow d_{k+1}(Y_2) $.
Let $Y_1' \in \mathcal{O}(Y_2)$.
If $d_{n-m-k}(Y_1')\searrow d_{n+1-m-k}(Y_1')$ and $d_{k}(Y_1')\searrow d_{k+1}(Y_1')$, then by Lemma \ref{lemma:bulge_hook}, we have
\[
\boldsymbol{d}(Y_2)\xrightarrow{n+1-m-k,n-m-k}\boldsymbol{d}(Y')\xrightarrow{k+1,k}\boldsymbol{d}(Y_1')\]
or 
\[\boldsymbol{d}(Y_2)\xrightarrow{k+1,k}\boldsymbol{d}(Y')\xrightarrow{n+1-m-k,n-m-k}\boldsymbol{d}(Y_1')
\] 
for $Y'\in \mathcal{F}(Y_{m,n})$, which is a contradiction.
Thus there exists no option $Y_1' \in \mathcal{O}(Y_2)$ such that $d_{n-m-k}(Y_1')\searrow d_{n+1-m-k}(Y_1')$ and $d_{k}(Y_1')\searrow d_{k+1}(Y_1')$, which contradicts $Y_2 \rightarrow Y_1$.

Thus we have proved Lemma \ref{lemma:same_center}.
\end{proof}

\begin{proof}[Proof of Theorem \ref{th:game_cor}]
We have shown that the map $E:\mathcal{T}(Y_{m,n}) \to \mathcal{T}(Y_{m,n+1})$ is injective (see \eqref{eq:EC}), and $E(\mathcal{O}(Y)) \subseteq \mathcal{O}(E(Y))$ for $Y \in \mathcal{T}(Y_{m,n})$ (see \eqref{eq:ENCNE}).
Hence it remains to show that $E(\mathcal{O}(Y)) \supseteq \mathcal{O}(E(Y))$ for $Y \in \mathcal{T}(Y_{m,n})$ and $E(\mathcal{T}(Y_{m,n})) = \mathcal{T}(Y_{m,n+1})$.

We first show that $E(\mathcal{O}(Y)) \supseteq \mathcal{O}(E(Y))$.
Let $Y \in \mathcal{T}(Y_{m,n})$, and let $X \in \mathcal{O}(E(Y))$.
There exists $-m<l\le r<n$ such that 
\[
\boldsymbol{d}(Y)\xrightarrow{l,r}\boldsymbol{d}(X) \tag{a} \label{eq:a},
\]
or there exist $-m<l\le r<n$ and $X'\in \mathcal{F}(Y_{m,n})$ such that
\[
\boldsymbol{d}(Y)\xrightarrow{l,r}\boldsymbol{d}(X')\xrightarrow{n-m-r,n-m-l}\boldsymbol{d}(X) \tag{b} \label{eq:b}.
\]
By Lemma \ref{lemma:same_center}, we have $d_{c}(E(Y))\nearrow d_{c+1}(E(Y))$, and $r \not= c$.

In the first case \eqref{eq:a}, we get $\boldsymbol{d}(X)_{[n+1-m-r,n+1-m-l]}\notin \mathbb{D}_{m,n+1}$.
If $l = c+1$, then $d_{c}(X) \searrow d_{c+1}(X)$ and $d_{c}(X) > d_{c+1}(X)$, which contradicts Lemma \ref{lemma:same_center}.
If $l \not= c+1$, then there exist $m < l_0,r_0 < n$ such that $e_l(l_0) = l,e_r(r_0) = r$.
By Lemma \ref{lemma:hook_cor}, we have $\boldsymbol{d}(Y)_{[l_0,r_0]} \in \mathbb{D}_{m,n}$, and $ (\boldsymbol{d}(Y)_{[l_0,r_0]})_{[n-m-r_0,n-m-l_0]} \notin \mathbb{D}_{m,n}$.
Thus the Young diagram $Y'\in\mathcal{T}(Y_{m,n})$ whose diagonal expression is equal to  $\boldsymbol{d}(Y)_{[l_0,r_0]} \in \mathbb{D}_{m,n}$ is an option of $Y$.
By the proof of Lemma \ref{lemma:step_cor}, we obtain $X = E(Y') \in E(\mathcal{O}(Y))$.

In the second case \eqref{eq:b}, if $l = c+1$, then $\boldsymbol{d}(X) = (\boldsymbol{d}(E(Y))_{[c+1,r]})_{[n+1-m-r,c]} = \boldsymbol{d}(E(Y))_{[n+1-m-r,r]}$.
Then there exist $m < l_0,r_0 < n$ such that $e_l(l_0) = n+1-m-r,e_r(r_0) = r$.
By Lemma \ref{lemma:e_lr}, we have $e_l(n-m-r_0)=e_l(n-m-r_0)+e_r(r_0)-r= n-m+1-r = e_l(l_0)$, and hence $l_0 = n-m-r_0$.
By Lemma \ref{lemma:hook_cor}, we have $\boldsymbol{d}(Y)_{[l_0,r_0]} \in \mathbb{D}_{m,n}$, and $ (\boldsymbol{d}(Y)_{[l_0,r_0]})_{[n-m-r_0,n-m-l_0]} = \boldsymbol{d}(Y)_{[l_0,r_0]})_{[l_0,r_0]} \notin \mathbb{D}_{m,n}$.
Thus the Young diagram $Y'\in\mathcal{T}(Y_{m,n})$ whose diagonal expression is equal to  $\boldsymbol{d}(Y)_{[l_0,r_0]} \in \mathbb{D}_{m,n}$ is an option of $Y$.
By the proof of Lemma \ref{lemma:step_cor}, we obtain $X = E(Y') \in E(\mathcal{O}(Y))$.
If $l \not= c+1$, then there exist $m < l_0,r_0 < n$ such that $e_l(l_0) = l,e_r(r_0) = r$.
By Lemma \ref{lemma:hook_cor}, we have $\boldsymbol{d}(Y)_{[l_0,r_0]} \in \mathbb{D}_{m,n}$ and $ (\boldsymbol{d}(Y)_{[l_0,r_0]})_{[n-m-r_0,n-m-l_0]} \in \mathbb{D}_{m,n}$.
Thus the Young diagram $Y'\in\mathcal{T}(Y_{m,n})$ whose diagonal expression is equal to  $(\boldsymbol{d}(Y)_{[l_0,r_0]})_{[n-m-r_0,n-m-l_0]} \in \mathbb{D}_{m,n}$ is an option of $Y$.
By the proof of Lemma \ref{lemma:step_cor}, we obtain $X = E(Y') \in E(\mathcal{O}(Y))$.

In any case, we obtain $X \in E(\mathcal{O}(Y))$, as desired.

We next show that $E(\mathcal{T}(Y_{m,n})) = \mathcal{T}(Y_{m,n+1})$.
Let $X' \in \mathcal{T}(Y_{m,n+1})$, and let $Y_{m,n+1} = X_0 \xrightarrow{}X_1\xrightarrow{}\cdots\xrightarrow{} X_k = X'$ be a transition from $Y_{m,n+1}$ to $X'$ in MHRG$(m,n+1)$.
We show by induction on $k$ that $X' \in E(\mathcal{T}(Y_{m,n}))$.
If $k=0$, then $X' = Y_{m,n+1} = E(Y_{m,n}) \in E(\mathcal{T}(Y_{m,n}))$.
Assume that $k>0$; note that $X_{k-1} \in E(\mathcal{T}(Y_{m,n}))$ by the induction hypothesis.
Let $X_{k-1}' \in \mathcal{T}(Y_{m,n})$ be such that $X_{k-1} = E(X_{k-1}')$.
Because $E(\mathcal{O}(X_{k-1}')) = \mathcal{O}(E(X_{k-1}')) = \mathcal{O}(X_{k-1})$ as shown above, we get $X' = X_k \in \mathcal{O}(X_{k-1}) = E(\mathcal{O}(X_{k-1}')) \subset E(\mathcal{T}(Y_{m,n}))$, as desired.
Therefore, we conclude that $E(\mathcal{T}(Y_{m,n})) \supset \mathcal{T}(Y_{m,n+1})$, and hence $E(\mathcal{T}(Y_{m,n})) = \mathcal{T}(Y_{m,n+1})$.

This completes the proof of Theorem \ref{th:game_cor}.
\end{proof}


\subsection{Case of MHRG$(1,n)$}
Recall that $Y_{1,l}=(l)$ denotes the rectangular Young diagram of size $1\times l$ for $l\in \mathbb{N}_0$.
\begin{theorem}\label{th:1n}
Let $m=1$, and $n\in \mathbb{N}$.
In MHRG$(1,n)$, 
\begin{center}
$\mathcal{T}(Y_{1,n})=\left \{ 
\begin{array}{cc}
    \mathcal{F}(Y_{1,n})  & \text{if $n$ is odd},\\ 
    \mathcal{F}(Y_{1,n})\setminus\{(\frac{n}{2})\} & \text{if $n$ is even}.
\end{array}
\right.$
\end{center}
Moreover, for $0\leq l\leq n$ such that $Y_{1,l}\in \mathcal{T}(Y_{1,n})$,
\begin{center}
$\mathcal{G}((l))=\left \{ 
\begin{array}{ccc}
    l  & \text{if $n$ is odd},\\ 
    l & \text{if $n$ is even and $l<n/2$},\\
    l-1 & \text{if $n$ is even and $n/2<l$}.
\end{array}
\right.$
\end{center}
In particular,
\begin{center}
$\mathcal{G}((n))=\left \{ 
\begin{array}{cc}
    n  & \text{if $n$ is odd},\\ 
    n-1 & \text{if $n$ is even}.
\end{array}
\right.$
\end{center}
\end{theorem}
\begin{proof}

By Theorem \ref{th:game_cor}, it suffices to show the assertion in the case that $n$ is odd.
We set $k=(n+1)/2\in\mathbb{N}$.
We see that for $0\leq l\leq n$, the unimodal numbering of $(l)\in \mathcal{F}(Y_{1,n})$ is as follows: 
\ytableausetup{centertableaux, mathmode, boxsize=2.0em}
\begin{center}
\begin{ytableau}
    1&2&\cdots&\scriptstyle l-1&l\\
\end{ytableau}
\quad if $0\leq l\leq k$,\\
\vspace{11pt}
\begin{ytableau}
    1&2&\cdots&\scriptstyle k-1&k&\scriptstyle k-1&\cdots&\scriptscriptstyle n+2-l&\scriptscriptstyle n+1-l\\
\end{ytableau}
\quad if $k <l\leq n$.
\end{center}
By this fact, we deduce that in MHRG$(1,n)$ (with odd n), the operation removing two hooks never happen.
Hence, we obtain $\mathcal{O}{(l)} = \{(i) \mid 0 \le i < l\}$ for all $0\leq l\leq n$.
The assertion of the theorem follows immediately from this fact and the definition of the $\mathcal{G}$-value (cf. Nim with one heap; see \cite[Chapter I]{Siegel}).
\end{proof}


\subsection{Case of MHRG$(2,n)$}

Let $m = 2$ and $n\ge 2$.
Recall that $Y = (\lambda_1,\lambda_2)$ denotes the Young diagram having $\lambda_1$ boxes in the 1st row, and $\lambda_2$ boxes in the 2nd row.
If $n$ is even, then MHRG$(2,n)$ is isomorphic to MHRG$(2,n+1)$ (see Theorem \ref{th:game_cor}).
Thus it suffices to study the case that $n$ is even; we set $n' := n/2 \in \mathbb{N}$.

\begin{lemma}\label{lemma:bulge_part}
Let $(\lambda_1,\lambda_2) \in \mathcal{F}(Y_{2,2n'})$, and let $\mu_1,\mu_2\in \mathbb{N}_0$ with $2n' \ge \mu_1 \ge \mu_2 \ge 0$.
Then $(\lambda_1,\lambda_2) = (\mu_1,\mu_2)$ if and only if $d_{\mu_1-1}(Y) \searrow d_{\mu_1}(Y)$, $d_{\mu_2-2}(Y) \searrow d_{\mu_2-1}(Y)$, and $d_{k-1}(Y) \nearrow d_{k}(Y)$ for $-2 < k \le 2n' = n$ with $k \not= \mu_1,\mu_2-1$.
\end{lemma}
\begin{proof}
If $\lambda_2 = 0$, then $(\lambda_1,0) = (\mu_1,\mu_2)$ if and only if $\mu_2 = 0$, $d_{-1}(Y) = 0$, $d_{k}(Y) = 1$ for $0\le k < \mu_1$, and $d_{k}(Y) = 0$ for $\mu_1\le k < 2n'$, which is equivalent to $d_{\mu_1-1}(Y) \searrow d_{\mu_1}(Y)$, $d_{\mu_2-2}(Y) \searrow d_{\mu_2-1}(Y)$, and $d_{k-1}(Y) \nearrow d_{k}(Y)$ for $-2 < k \le 2n' = n$ with $k \not= \mu_1,\mu_2-1$.

If $\lambda_2 > 0$, then $(\lambda_1,\lambda_2) = (\mu_1,\mu_2)$ if and only if $d_{-1}(Y) = 1$, $d_{k}(Y) = 2$ for $0\le k < \mu_2-1$, $d_{k}(Y) = 1$ for $\mu_2-1\le k < \mu_1$, and $d_{k}(Y) = 0$ for $\mu_1\le k < 2n'$, which is equivalent to $d_{\mu_1-1}(Y) \searrow d_{\mu_1}(Y)$, $d_{\mu_2-2}(Y) \searrow d_{\mu_2-1}(Y)$, and $d_{k-1}(Y) \nearrow d_{k}(Y)$ for $-2 < k \le 2n' = n$ with $k \not= \mu_1,\mu_2-1$.

Thus we have proved the lemma.
\end{proof}

\begin{lemma}\label{lemma:2n_step}
Let $Y = (\lambda_1,\lambda_2) \in \mathcal{F}(Y_{2,2n'})$.
Let $(i,j) \in Y$, and set $Y' = (\lambda_1',\lambda_2') = Y\setminus h_Y{(i,j)}$.
Then, $\lambda_1'+\lambda_2'=2n'$ if and only if there exists a box $(i',j') \in Y'$ such that $\mathcal{A}(h_Y(i,j)) = \mathcal{A}(h_{Y'}(i',j'))$.
In this case, $Y'':= Y'\setminus h_Y'(i',j')$ is equal to $(2n'-\lambda_2,2n'-\lambda_1)$.
\end{lemma}
\begin{proof}
We first show the ``if'' part.
By Lemma \ref{lemma:hook_and_DRep}, there exist $-2 < l,r < 2n'$ such that $\boldsymbol{d}(Y)\xrightarrow{l,r}\boldsymbol{d}(Y')$.
If there exists a box $(i',j') \in Y'$ such that $\mathcal{A}(h_Y(i,j)) = \mathcal{A}(h_{Y'}(i',j'))$, then it follows from Lemma \ref{lemma:same_number_hook} that $\boldsymbol{d}(Y')_{[2n'-2-r,2n'-2-l]} \in \mathbb{D}_{2,2n'}$.
Note that $2n'-2-l+1 \not= l$.
By Lemmas \ref{lemma:bulge_hook} and \ref{lemma:bulge_part}, the pair $(r,2n'-2-l)$ is equal to $(\lambda_1-1,\lambda_2-2)$ or $(\lambda_2-2,\lambda_1-1)$, and hence $\lambda_1'+\lambda_2'=(\lambda_1+\lambda_2)-(r-l+1)= (\lambda_1+\lambda_2)-(2-2n'+\lambda_1-1+\lambda_2-2+1)= (\lambda_1+\lambda_2)-(2n'+\lambda_1+\lambda_2) =2n'$.

We next show the ``only if'' part.
As above, assume that $\boldsymbol{d}(Y)\xrightarrow{l,r}\boldsymbol{d}(Y')$.
If $\lambda_1'+\lambda_2'=2n'$, then $d_{\lambda_1'-1}(Y') \searrow d_{\lambda_1'}(Y'),d_{2n'-\lambda_1'-2}(Y') \searrow d_{2n'-\lambda_1-1}(Y')$, and $d_{k-1}(Y') \nearrow d_{k}(Y')$ for $-2 < k \le 2n'$ with $k \not= \lambda_1',2n'-\lambda_1'-1$.
By Lemma \ref{lemma:bulge_hook}, we have $l = \lambda_1'$ or $l = 2n'-\lambda_1'-1$.
If $l = \lambda_1'$, then $r \not= 2n'-\lambda_1'-2$, and hence $2n'-2-r \not = \lambda_1'$.
Thus, $\boldsymbol{d}(Y')_{[2n'-2-r,2n'-2-l]} =\boldsymbol{d}(Y')_{[2n'-2-r,2n'-2-\lambda_1']} \in \mathbb{D}_{2,2n'}$.
If $l = 2n'-\lambda_1'-1$, then $r \not= \lambda_1'-1$, and hence $2n'-2-r \not = 2n'-1-\lambda_1'$.
Thus, $\boldsymbol{d}(Y')_{[2n'-2-r,2n'-2-l]} = \boldsymbol{d}(Y')_{[2n'-2-r,\lambda_1'-1]} \in \mathbb{D}_{2,2n'}$.
In both cases, we have $\boldsymbol{d}(Y')_{[2n'-2-r,2n'-2-l]} \in \mathbb{D}_{2,2n'}$, which implies that there exists a box $(i',j') \in Y'$ such that $\mathcal{A}(h_Y(i,j)) = \mathcal{A}(h_{Y'}(i',j'))$ (see Lemma \ref{lemma:same_number_hook}).

Finally, let us show that $Y'':= Y'\setminus h_{Y'}(i',j')$ is equal to $(2n'-\lambda_2,2n'-\lambda_1)$.
By Lemma \ref{lemma:bulge_hook}, we have
\[
\boldsymbol{d}(Y)\xrightarrow{l,r}\boldsymbol{d}(Y')\xrightarrow{2n'-2-r,2n'-2-l}\boldsymbol{d}(Y''), \]
and $d_{l-1}(Y'')\searrow d_{l}(Y''),d_{2n'-2-r-1}(Y'') \searrow d_{2n'-2-r}(Y'')$, and $d_{k}(Y'') \nearrow d_{k}(Y'')$ for $-2 < k \le 2n'$ with $k \not= l,2n'-2-r$.
As seen above, the pair $(r,2n'-2-l)$ is equal to $(\lambda_1-1,\lambda_2-2)$ or $(\lambda_2-2,\lambda_1-1)$.
If $(r,2n'-2-l) = (\lambda_1-1,\lambda_2-2)$, then $l = 2n'-\lambda_2 > 2n'-1-\lambda_1 = 2n'-2-l$.
If $(r,2n'-2-l) = (\lambda_2-2,\lambda_1-1)$, then $l = 2n'-1-\lambda_1 < 2n'-\lambda_2 = 2n'-2-l$.
In both cases, we get $d_{2n'-2-\lambda_1}(Y'') \searrow d_{2n'-1-\lambda_1}(Y''),d_{2n'-\lambda_2-1}(Y'') \searrow d_{2n'-\lambda_2}(Y'')$, and $d_{k}(Y'') \nearrow d_{k}(Y'')$ for $-2 < k \le 2n'$ with $k \not= 2n'-1-\lambda_1,2n'-\lambda_2$.
Hence we obtain $Y'' = (2n'-\lambda_2,2n'-\lambda_1)$ by Lemma \ref{lemma:bulge_part}, as desired.
\end{proof}

For $Y \in \mathcal{F}(Y_{2,2n'})$, we set $OH(Y) := \{Y\setminus h_Y(i,j)\mid (i,j) \in Y\}$;
if $Y=(\lambda_1,\lambda_2)$, then
\begin{eqnarray*}
OH(Y) &=& \{(\lambda'_1,\lambda_2)\mid\lambda_2\le \lambda'_1<\lambda_1\} \cup \{(\lambda_1,\lambda'_2)\mid0\le \lambda'_2<\lambda_2\}\\
&\cup& \{(\lambda_2-1,\lambda'_1)\mid 0\le \lambda'_1<\lambda_2\}.
\end{eqnarray*}
By Lemma \ref{lemma:2n_step}, we can easily show the following lemma.

\begin{lemma}\label{lemma:2n_transition}
In MHRG$(2,2n')$,
\[
\mathcal{T}(Y_{2,2n'}) = \mathcal{F}(Y_{2,2n'})\setminus\{(\lambda'_1,\lambda'_2) \in \mathcal{F}(Y_{2,2n'})\mid\lambda'_1+\lambda'_2 = 2n'\}.
\]
Moreover, for $Y = (\lambda_1,\lambda_2) \in \mathcal{F}(Y_{2,2n'})$,
\begin{enumerate}
\item[(1)] if $\lambda_1+\lambda_2 < 2n'$ , then $\mathcal{O}(Y)=OH(Y)$;
\item[(2)] if $\lambda_1+\lambda_2 > 2n'$ , then $\mathcal{O}(Y) = OH(Y)\setminus\{(\lambda'_1,\lambda'_2) \in \mathcal{F}(Y_{2,2n'})\mid\lambda'_1+\lambda'_2 = 2n'\} \cup \{(2n'-\lambda_2,2n'-\lambda_1)\}$.
\end{enumerate}

\end{lemma}

By Lemma \ref{lemma:2n_transition}, the $\mathcal{G}$-value of $Y = (\lambda_1,\lambda_2) \in \mathcal{T}(Y_{2,2n'})$ with $\lambda_1 + \lambda_2 < n = 2n'$ is equal to the $\mathcal{G}$-value of the game position corresponding to $Y$ in the Sato-Welter game (see, e.g., \cite[Theorem 2]{Kawanaka}).
For later use, we list those $Y = (\lambda_1,\lambda_2) \in \mathcal{T}(Y_{2,2n'})$ with $\lambda_1 + \lambda_2 < 2n'$ whose $\mathcal{G}$-values are $0$, $1$, or $2$.
\begin{table}[htb]
\begin{center}
\begin{tabular}{|c|c|c|}\hline
    $\mathcal{G}(Y) = 0$ & $\mathcal{G}(Y) = 1$ & $\mathcal{G}(Y) = 2$ \\ \hline \hline
    $(2i,2i)$ &  \begin{tabular}{c}$(1+4i,4i)$\\$(2+4i,1+4i)$\end{tabular} & \begin{tabular}{c}$(2+4i,4i)$\\$(1+4i,1+4i)$\end{tabular} \\\hline
    \end{tabular}
\end{center}
\caption{$Y = (\lambda_1,\lambda_2) \in \mathcal{F}(Y_{2,2n'})$ with $\lambda_1 + \lambda_2 < 2n'$ whose $\mathcal{G}$-values are $0$, $1$, or $2$.}\label{table:012l}
\end{table}

\begin{theorem}\label{th:2n}
As above, assume that $n$ is even, and set $n' = n/2$.
In MHRG$(2,2n')$, the list of those $Y = (\lambda_1,\lambda_2) \in \mathcal{F}(Y_{2,2n'})$ with $\lambda_1 + \lambda_2 > 2n'$ whose $\mathcal{G}$-values are $0$,$1$ or $2$ is given by Table \ref{table:012u}.
\begin{table}[htb]\footnotesize
\begin{center}
\begin{tabular}{|c||c|c|c|}\hline
    $\!n'\!$  & $\mathcal{G}(Y) = 0$ & $\mathcal{G}(Y) = 1$ & $\mathcal{G}(Y) = 2$ \\ \hline \hline
    $\!4n''\!$ & \begin{tabular}{c}$(n'+1+4i,n'+4i)$ \\ $\!\!\!\!(n'+2+4i,n'+1+4i)\!\!\!\!$\end{tabular} &  \begin{tabular}{c}$(n'+2,n')$\\$(n'+1,n'+1)$ \\ $\!\!\!\!(n'+4+2i,n'+4+2i)\!\!\!\!$\end{tabular} & \begin{tabular}{c}$(n'+2,n'+2)$\\ \!$(n'+3,n')$\! \\ \!$(n'+4,n'+1)$\! \\ $\!\!\!\!(n'+7+4i,n'+6+4i)\!\!\!\!$\\ $\!\!\!\!(n'+8+4i,n'+7+4i)\!\!\!\!$\end{tabular}\\ \hline
    $\!4n''+1\!$ & \begin{tabular}{c}$\!\!\!\!(n'+2+4i,n'+1+4i)\!\!\!\!$ \\ $\!\!\!\!(n'+3+4i,n'+2+4i)\!\!\!\!$\end{tabular} & \begin{tabular}{c}$(n'+2+2i,n'+2i)$\end{tabular} & \begin{tabular}{c}$(n'+1,n')$\\ \!$(n'+2,n'-1)$\! \\ \!$(n'+3,n'+1)$\! \\ $\!\!\!\!(n'+5+2i,n'+5+2i)\!\!\!\!$\end{tabular}\\ \hline
    $\!4n''+2\!$ & \begin{tabular}{c}$(n'+1+4i,n'+4i)$ \\ $\!\!\!\!(n'+2+4i,n'+1+4i)\!\!\!\!$\end{tabular} & \begin{tabular}{c}$\!\!\!\!(n'+2+2i,n'+2+2i)\!\!\!\!$\end{tabular} & \begin{tabular}{c}$\!\!\!\!(n'+3+4i,n'+2+4i)\!\!\!\!$ \\ $\!\!\!\!(n'+4+4i,n'+3+4i)\!\!\!\!$\end{tabular}\\ \hline
    $\!4n''+3\!$ & \begin{tabular}{c}$\!\!\!\!(n'+2+4i,n'+1+4i)\!\!\!\!$ \\ $\!\!\!\!(n'+3+4i,n'+2+4i)\!\!\!\!$\end{tabular} & \begin{tabular}{c}$\!\!\!\!(n'+1+2i,n'+1+2i)\!\!\!\!$\end{tabular} & \begin{tabular}{c}$\!\!\!\!(n'+4+8i,n'+1+8i)\!\!\!\!$ \\ $\!\!\!\!(n'+5+8i,n'+2+8i)\!\!\!\!$ \\ $\!\!\!\!(n'+6+8i,n'+3+8i)\!\!\!\!$ \\ $\!\!\!\!(n'+7+8i,n'+4+8i)\!\!\!\!$\end{tabular} \\ \hline
\end{tabular}
\end{center}
\caption{$Y = (\lambda_1,\lambda_2) \in \mathcal{F}(Y_{2,2n'})$ with $\lambda_1 + \lambda_2 > 2n'$ whose $\mathcal{G}$-values are $0$, $1$, or $2$.}\label{table:012u}
\end{table}
\end{theorem}
\begin{proof}
We give a proof only for the case of $n' = 4n''$ for $n'' \in \mathbb{N}$; the proofs of the cases $n' = 4n''+1,4n''+2,4n''+3$ for $n'' \in \mathbb{N}_0$ are similar.
We set $G_k := \{(\lambda_1,\lambda_2) \in \mathcal{T}(Y_{2,2n'}) \mid \lambda_1 + \lambda_2 > 2n',\mathcal{G}((\lambda_1,\lambda_2)) = k \}$ for $k \in \mathbb{N}_0$.

First, we determine $G_0$, following the definition of the $\mathcal{G}$-value.
Let $Y = (\lambda_1,\lambda_2) \in \mathcal{T}(Y_{2,2n'})$ with $\lambda_1 + \lambda_2 > 2n'$.
If $\lambda_2 < n'$, and $\lambda_2$ is even (resp., odd), then we deduce that $Y' = (\lambda_2,\lambda_2)$ (resp., $Y' = (\lambda_2-1,\lambda_2-1)$) is contained in $\mathcal{O}(Y)$.
Because $\mathcal{G}(Y') = 0$ by Table \ref{table:012l}, we obtain $Y\notin G_0$.

Now, we see by Lemma \ref{lemma:2n_transition} that
\begin{align*}
\mathcal{O}((n'+1,n')) &=\Bigl(\{(n',n')\} \cup \{(n'+1,\lambda'_2)\mid0\le \lambda'_2<n'\} \\
&\quad \cup \{(n'-1,\lambda'_1)\mid 0\le \lambda'_1<n'\}\Bigr) \\
&\quad\setminus\{(\lambda'_1,\lambda'_2)\mid\lambda'_1+\lambda'_2= 2n'\} \cup \{(n',n'-1)\} \\
&= \{(n'+1,\lambda'_2)\mid0\le \lambda'_2<n'-1\} \\
&\quad \cup \{(n'-1,\lambda'_1)\mid 0\le \lambda'_1<n'\} \cup \{(n',n'-1)\}.
\end{align*}
Note that $n' = 4n''$ is even.
By Table \ref{table:012l} and the argument above, it can be checked that $\mathcal{O}((n'+1,n'))$ has no position whose $\mathcal{G}$-value is $0$.
Thus we get $\mathcal{G}((n'+1,n')) = 0$.
If $Y \in \{(n'+1,n'+1)\} \cup \{(\lambda_1',n')\mid n'+2\le \lambda_1'\le 2n'\} \cup \{(\lambda_1',n'+2)\mid n'+2\le \lambda_1'\le 2n'\}$, then $(n'+1,n') \in \mathcal{O}(Y)$, which implies that $Y\notin G_0$.

Similarly, we see by Lemma \ref{lemma:2n_transition} that
\begin{align*}
\mathcal{O}((n'+2,n'+1)) &=\Bigl(\{(n'+1,n'+1)\} \cup \{(n'+2,\lambda'_2)\mid0\le \lambda'_2<n'+1\} \\
&\quad \cup \{(n',\lambda'_1)\mid 0\le \lambda'_1<n'+1\}\Bigr) \\
&\quad\setminus\{(\lambda'_1,\lambda'_2)\mid\lambda'_1+\lambda'_2 = 2n'\} \cup \{(n'-1,n'-2)\} \\
&= \{(n'+1,n'+1)\} \\
&\quad \cup \Bigl(\{(n'+2,\lambda'_2)\mid0\le \lambda'_2<n'+1\}\setminus\{(n'+2,n'-2)\} \Bigr) \\
&\quad \cup \{(n',\lambda'_1)\mid 0\le \lambda'_1<n'\} \cup \{(n'-1,n'-2)\}.
\end{align*}
By Table \ref{table:012l} and the argument above, we deduce that $\mathcal{O}((n'+2,n'+1))$ has no position whose $\mathcal{G}$-value is $0$.
Thus we get $\mathcal{G}((n'+2,n'+1)) = 0$.
If $Y \in \{(n'+2,n'+2)\} \cup \{(\lambda_1',n'+1)\mid n'+3\le \lambda_1'\le 2n'\} \cup \{(\lambda_1',n'+3)\mid n'+3\le \lambda_1'\le 2n'\}$, then $(n'+2,n'+1) \in \mathcal{O}(Y)$, which implies that $Y\notin G_0$.
Therefore, for $Y = (\lambda_1,\lambda_2) \in \mathcal{F}(Y_{2,2n'})$ with $n' \le \lambda_2 \le n'+3$ and $\lambda_2 \le \lambda_1 \le 2n'$, 
\begin{equation}\label{eq:g0}
Y\in G_0\text{ if and only if }Y=(n'+1,n'),(n'+2,n'+1).
\end{equation}

Let $i\in\mathbb{N}$ with $n'+4+4i \le 2n'$.
By Lemma \ref{lemma:2n_transition}, $(n'+4+4i,n'+4+4i) \rightarrow (n'-4-4i,n'-4-4i)$.
Since $\mathcal{G}((n'-4-4i,n'-4-4i)) = \mathcal{G}((4n''-4-4i,4n''-4-4i)) = 0$ by Table \ref{table:012l}, we get $\mathcal{G}((n'+4+4i,n'+4+4i)) \not= 0$.
Furthermore, in exactly the same way as \eqref{eq:g0}, it can be verified that for $Y = (\lambda_1,\lambda_2) \in \mathcal{F}(Y_{2,2n'})$ with $n'+4i \le \lambda_2 \le n'+3+4i$ and $\lambda_2 \le \lambda_1 \le 2n'$, $Y \in G_0$ if and only if $Y=(n'+1+4i,n'+4i),(n'+2+4i,n'+1+4i)$.
Therefore, we obtain 
\begin{align*}
G_0 &= \Bigl(\{(n'+1+4i,n'+4i) \mid i \ge 0 \} \cup \{(n'+2+4i,n'+1+4i) \mid i \ge 0 \}\Bigr)\\
&\quad \cap \mathcal{F}(Y_{2,2n'}),
\end{align*}
as desired.

Next, we determine $G_1$.
Let $Y = (\lambda_1,\lambda_2) \in \mathcal{T}(Y_{2,2n'})$ with $\lambda_1 + \lambda_2 > 2n'$.
In a way similar to $G_0$, if $\lambda_2 < n'$, then $Y\notin G_1$.
By Table \ref{table:012l} and $\mathcal{G}((n'+1,n')) = 0$,  we deduce that $\mathcal{O}((n'+2,n'))$ and $\mathcal{O}((n'+1,n'+1))$ have no position whose $\mathcal{G}$-value is $1$, and have a position $(n'+1,n')$, whose $\mathcal{G}$-value is $0$.
Thus we get $\mathcal{G}((n'+2,n')) = \mathcal{G}((n'+1,n'+1)) = 1$.
If $Y \in \{(\lambda_1',n')\mid n'+2\le \lambda_1'\le 2n'\} \cup\{(\lambda_1',n'+1)\mid n'+1\le \lambda_1'\le 2n'\} \cup \{(\lambda_1',n'+2)\mid n'+1\le \lambda_1'\le 2n'\} \cup \{(\lambda_1',n'+3)\mid n'+2\le \lambda_1'\le 2n'\}$, then $(n'+2,n') \in \mathcal{O}(Y)$ or $(n'+1,n'+1) \in \mathcal{O}(Y)$, which implies that $Y\notin G_1$.
Therefore, for $Y = (\lambda_1,\lambda_2) \in \mathcal{F}(Y_{2,2n'})$ with $n' \le \lambda_2 \le n'+3$ and $\lambda_2 \le \lambda_1 \le 2n'$, $Y\in G_1$ if and only if $Y=(n'+2,n'),(n'+1,n'+1)$.

We see by Lemma \ref{lemma:2n_transition} that
\begin{align*}
\mathcal{O}((n'+4,n'+4)) &=\Bigl(\{(n'+4,\lambda'_2)\mid0\le \lambda'_2<n'+3\} \\
&\quad \cup \{(n'+3,\lambda'_1)\mid 0\le \lambda'_1<n'+3\}\Bigr) \\
&\quad\setminus\{(\lambda'_1,\lambda'_2)\mid\lambda'_1+\lambda'_2 = 2n'\} \cup \{(n'-4,n'-4)\}.
\end{align*}
By Table \ref{table:012l} and the argument above, we deduce that $\mathcal{O}((n'+2,n'+1))$ has no position whose $\mathcal{G}$-value is $1$, and have a position $(n'-4,n'-4)$, whose $\mathcal{G}$-value is $0$.
Thus we get $\mathcal{G}((n'+4,n'+4)) = 1$.
If $Y \in \{(\lambda_1',n'+4)\mid n'+5\le \lambda_1'\le 2n'\} \cup \{(\lambda_1',n'+5)\mid n'+5\le \lambda_1'\le 2n'\})$, then $(n'+4,n'+4) \in \mathcal{O}(Y)$, which implies that $Y\notin G_1$.
Therefore, for $Y = (\lambda_1,\lambda_2) \in \mathcal{F}(Y_{2,2n'})$ with $n'+4 \le \lambda_2 \le n'+5$ and $\lambda_2 \le \lambda_1 \le 2n'$, $Y\in G_1$ if and only if $Y=(n'+4,n'+4)$.
Similarly, for each $i\in \mathbb{N}$ (with $n'+4+2i\le 2n'$), it can be verified that for $Y = (\lambda_1,\lambda_2) \in \mathcal{F}(Y_{2,2n'})$ with $n'+4+2i \le \lambda_2 \le n'+5+2i$ and $\lambda_2 \le \lambda_1 \le 2n'$, $Y \in G_1$ if and only if $Y=(n'+4+2i,n'+4+2i)$.
Therefore, we obtain 
\begin{align*}
G_1 &= \Bigl(\{(n'+2,n'), (n'+1,n'+1)\} \cup \{(n'+4+2i,n'+4+2i) \mid i \ge 0 \}\Bigr)\\
&\quad \cap \mathcal{F}(Y_{2,2n'}),
\end{align*}
as desired.

Finally, we determine $G_2$.
Let $Y = (\lambda_1,\lambda_2) \in \mathcal{T}(Y_{2,2n'})$ with $\lambda_1 + \lambda_2 > 2n'$.
In the same manner as for $G_0$ and $G_1$, we determine $G_2$ as follows.
\begin{itemize}
    \item If $\lambda_2 < n'$, then $Y\notin G_2$.
    \item If $n' \le \lambda_2 \le n'+5$ and $\lambda_2 \le \lambda_1 \le 2n'$, then $Y\in G_1$ if and only if $Y=(n'+2,n'+2),(n'+3,n'),(n'+4,n'+1)$.
    \item For each $i\in \mathbb{N}_0$ (with $n'+6+4i\le 2n'$), if $n'+6+4i \le \lambda_2 \le n'+9+4i$ and $\lambda_2 \le \lambda_1 \le 2n'$, then $Y \in G_0$ if and only if $Y=(n'+7+4i,n'+6+4i),(n'+8+4i,n'+7+4i)$.
\end{itemize}
Therefore, we obtain 
\begin{align*}
G_2 &= \Bigl(\{(n'+2,n'+2),(n'+3,n'),(n'+4,n'+1)\} \\
&\quad \cup \{(n'+7+4i,n'+6+4i) \mid i \ge 0 \} \\
&\quad \cup \{(n'+8+4i,n'+7+4i) \mid i \ge 0 \}\Bigr) \cap \mathcal{F}(Y_{2,2n'}), 
\end{align*}
as desired.
This complete the proof of Theorem \ref{th:2n}.
\end{proof}

The following is an immediate consequence of Theorem \ref{th:2n}, together with theorem \ref{th:game_cor}.

\begin{corollary}\label{cor:2n+1}
Let $n \ge 2$.
In MHRG$(2,n)$, the $\mathcal{G}$-value of the starting position $Y_{2,n}$ is given as follows: 
\begin{equation*}
\mathcal{G}(Y_{2,n}) = \left\{
\begin{aligned}
    3\ &\text{ if } n = 2,3,\\
    2\ &\text{ if } n \equiv 2,3 \text{ mod } 8,\\
    1\ &\text{ otherwise}.
\end{aligned}\right.
\end{equation*}
\end{corollary}
\begin{proof}
We can easily calculate the $\mathcal{G}$-value of the starting position in the cases that $n = 2,3$.
In the other case, we can prove by Theorem \ref{th:2n} and Theorem \ref{th:game_cor}.
\end{proof}


\section{Relation between MHRG\ and\ HRG in terms of Shifted Young Diagrams}
\subsection{Shifted Young Diagrams}
Let us recall shifted Young diagrams (for the details, see, e.g., \cite{Olsson}).
Let $m\in \mathbb{N}$, and let $\lambda_1,\ldots,\lambda_m \in \mathbb{N}$ be such that $\lambda_1> \cdots>\lambda_m >0$.
The set $S=(\lambda_1,\ldots,\lambda_m):= \{(i,j)\in \mathbb{N}^2\mid i\leq j,\ 1\le i \le m,\ 1\le j \le \lambda_i\}$
is called the shifted Young diagram corresponding to $(\lambda_1,\ldots,\lambda_m)$.
An element of the shifted Young diagram is called a box, and the shifted Young diagram is described in terms of boxes as follows.

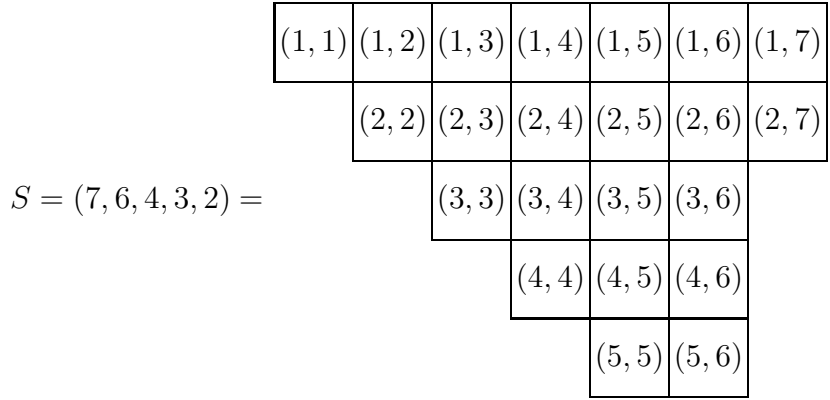
\ytableausetup{centertableaux, mathmode, boxsize=2.5em}
\begin{center}
$S=(7,6,4,3,2)=$
\begin{ytableau}
    (1,1)&(1,2)&(1,3)&(1,4)&(1,5)&(1,6)&(1,7) \\
    \none&(2,2)&(2,3)&(2,4)&(2,5)&(2,6)&(2,7) \\
    \none&\none&(3,3)&(3,4)&(3,5)&(3,6) \\
    \none&\none&\none&(4,4)&(4,5)&(4,6) \\
    \none&\none&\none&\none&(5,5)&(5,6) \\
\end{ytableau}
\end{center}
For $i\in \mathbb{N}$, the subset $\{(i,j)\mid j\in \mathbb{N}\}\cap S$ of $S$ is called the $i$-th row of $S$.
Similarly, for $j\in \mathbb{N}$, the subset $\{(i,j)\mid i\in \mathbb{N}\}\cap S$ of $S$ is called the $j$-th column of $S$. 
We call $h(S):=\mathrm{max}\{i \mid (i,j) \in S\}$ the height of $S$.

For a shifted Young diagram $S$, let $\mathcal{F}(S)$ denote the set of all shifted Young diagrams contained in $S$.

\subsection{Hooks of a Shifted Young Diagram}

\begin{definition}\label{def:Shifted_Hook}
For a box $(i,j)$ of a shifted Young diagram $S$, we define 
\begin{eqnarray*}
    \mathrm{arm}_S{(i,j)}&:=& \{(i',j')\in S\mid i=i', j<j'\},\\
    \mathrm{leg}_S{(i,j)}&:=& \{(i',j')\in S\mid i<i', j=j'\},\\
    \mathrm{tail}_S{(i,j)}&:=& \{(i',j')\in S\mid j+1=i', j<j'\},\\
    h_S{(i,j)}&:=& \{(i,j)\}\sqcup \mathrm{arm}_S{(i,j)}\sqcup \mathrm{leg}_S{(i,j)}\sqcup \mathrm{tail}_S{(i,j)}.
\end{eqnarray*}
The set $h_S{(i,j)}$ is called the hook corresponding to the box $(i,j)$.
\end{definition}

\begin{example}\label{ex:shifted_hooks}
In the figures below, the shadowed boxes form the hook corresponding to the box $v=(i,j)$. 
\ytableausetup{centertableaux, mathmode, boxsize=1.0em}
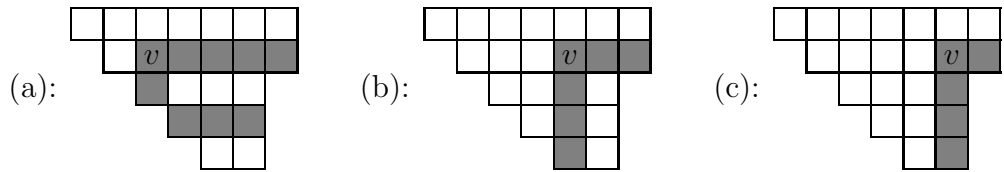
\begin{center}
(a):
\begin{ytableau}
    \ &\ &\ &\ &\ &\ &\ \\
    \none&\ &*(gray)v &*(gray) &*(gray) &*(gray) &*(gray) \\
    \none&\none&*(gray) &\ &\ &\ \\
    \none&\none&\none&*(gray) &*(gray) &*(gray) \\
    \none&\none&\none&\none&\ &\ \\
\end{ytableau}\qquad
(b):
\begin{ytableau}
    \ &\ &\ &\ &\ &\ &\ \\
    \none&\ &\ &\ &*(gray)v &*(gray) &*(gray) \\
    \none&\none&\ &\ &*(gray) &\ \\
    \none&\none&\none&\ &*(gray) &\ \\
    \none&\none&\none&\none&*(gray) &\ \\
\end{ytableau}\qquad
(c):
\begin{ytableau}
    \ &\ &\ &\ &\ &\ &\ \\
    \none&\ &\ &\ &\ &*(gray)v &*(gray) \\
    \none&\none&\ &\ &\ &*(gray) \\
    \none&\none&\none&\ &\ &*(gray) \\
    \none&\none&\none&\none&\ &*(gray) \\
\end{ytableau}
\end{center}
\end{example}

\begin{definition}\label{def:Shifted_Hook_removing}
For a box $(i,j)$ of a shifted Young diagram $S$, we remove the hook $h_S{(i,j)}$ corresponding to the box $(i,j)$ as follows:
\begin{enumerate}
\item Remove all boxes in the hook $h_S{(i,j)}$.
\item Move each box $(i',j')$ satisfying $j+1 > i'>i$ and $j'>j$ to $(i'-1,j'-1)$.
\item Move each box $(i',j')$ satisfying $i'>j+1$ to $(i'-2,j'-2)$.
\end{enumerate}
\end{definition}

\begin{example}\label{ex:Shifted_Hook_removing}
If we remove the hook corresponding to the box $(2,3)$ from the shifted Young diagram $S= (7,6,4,3,2)$, then we get $S'=(7,4,2)$.\\
\end{example}

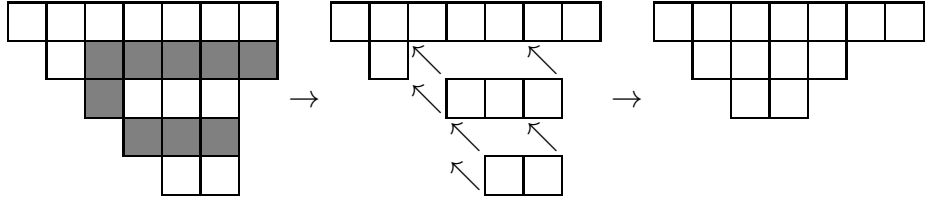
\ytableausetup{centertableaux, mathmode, boxsize=1.2em}
\begin{center}
\begin{ytableau}
    \ &\ &\ &\ &\ &\ &\ \\
    \none&\ &*(gray)\ &*(gray)\ &*(gray)\ &*(gray)\ &*(gray)\ \\
    \none&\none&*(gray)\ &\ &\ &\ \\
    \none&\none&\none&*(gray)\ &*(gray)\ &*(gray)\ \\
    \none&\none&\none&\none&\ &\ \\
\end{ytableau}
$\to$
\begin{ytableau}
    \ &\ &\ &\ &\ &\ &\ \\
    \none&\ &\none[\nwarrow]&\none&\none&\none[\nwarrow]&\none \\
    \none&\none&\none[\nwarrow] &\ &\ &\ \\
    \none&\none&\none&\none[\nwarrow] &\none &\none[\nwarrow] \\
    \none&\none&\none&\none[\nwarrow]&\ &\ \\
\end{ytableau}
\vspace{5mm}
$\to$
\begin{ytableau}
    \ &\ &\ &\ &\ &\ &\ \\
    \none&\ &\ &\ &\ \\
    \none&\none&\ &\ \\
    \none\\
    \none\\
\end{ytableau}\\
\end{center}

\begin{definition}\label{HRG}
The rules of Hook Removing Game (HRG for short) in teams of shifted Young diagrams are as follows.
\begin{enumerate}
    \item[(HS1)] The game is played by two players.
    \item[(HS2)] The two players alternately make a move.
    \item[(HS3)] Given a shifted Young diagram $S$, the player chooses a box $(i,j)\in S$, and remove the hook $h_S{(i,j)}$ corresponding to the box $(i,j)$ from $S$. 
    \item[(HS4)] The player who makes the empty shifted Young diagram $\emptyset$ wins.
\end{enumerate}
\end{definition}

We denote HRG (in terms of shifted Young diagrams) whose starting position is a shifted Young diagram $S$ by HRG$(S)$.
It is clear from the definition of HRG$(S)$ that $\mathcal{F}(S)$ is identical to the set of all positions in HRG$(S)$.


\begin{proposition}\label{pr:shifted}
Let $S=(\lambda_1,\lambda_2,\ldots,\lambda_n)$ be a shifted Young diagram, and let $T$ be a shifted Young diagram containing $S$. 
The $\mathcal{G}$-value of $S$ in HRG$(T)$ is equal to
\begin{center}
$\mathcal{G}(S)=\displaystyle\bigoplus_{1\leq i\leq n}\lambda_i$,
\end{center}
where $\bigoplus_i^{}{a_i}$ denotes the nim-sum (the addition of numbers in binary form without carry) of all $a_i$'s.
\end{proposition}

This formula may be well-known for experts;
we deduce this formula from the results of \cite{Kawanaka}, or by the fact that HRG$(S)$ is isomorphic to Turning Turtles (for Turning Turtles, see, e.g., \cite[page 182]{Siegel}).


\subsection{Diagonal Expression of a Shifted Young Diagram}

We explain the diagonal expression for shifted Young diagrams.
Fix $n\in\mathbb{N}$. 
We write an element $\boldsymbol{b} \in \mathbb{N}_0^{n+1}$ as $\boldsymbol{b} = [b_0,\ldots,b_n]$.
We denote by $\mathbb{SD}_{n}\subset\mathbb{N}_0^{n+1}$ the set of all elements $\boldsymbol{b}=[b_0,\ldots,b_n] \in \mathbb{N}_0^{n+1}$ with $b_n= 0$ satisfying $0\leq b_k-b_{k+1}\leq 1\ \mathrm{for}\ 0\leq k <n$.

Let $S_n$ denote the shifted Young diagram corresponding to $(n,n-1,n-2,\ldots,2,1)$. 
Let $S\in\mathcal{F}(S_n)$.
We set $d_k = d_k(S) := \#\{(i,j)\in S\mid j-i=k\}$ for $k \in \mathbb{Z}$.
Note that if $k \ge n$, then $d_k = 0$.
As Proposition \ref{prop:DRep}, we deduce that $\boldsymbol{sd}(S)=\boldsymbol{sd}_n(S):= [d_{0}(S),\ldots,d_{n}(S)]\in \mathbb{SD}_{n}$ for $S \in \mathcal{F}(S_n)$, and the map $\boldsymbol{sd}=\boldsymbol{sd}_n:\mathcal{F}(S_n)\to \mathbb{SD}_{n}$, $S\mapsto \boldsymbol{sd}(S)$, is bijective.
\begin{definition}\label{def:Shifted_DRep}
We call $\boldsymbol{sd}(S)=\boldsymbol{sd}_n(S)$ the diagonal expression of $S\in \mathcal{F}(S_n)$.
\end{definition}

Let $\boldsymbol{b} = [b_0,\ldots, b_n] \in \mathbb{SD}_{n}$, $\boldsymbol{b'} = [b'_0,\ldots, b'_n] \in \mathbb{N}_0^{n+1}$, and $0 \leq l \le r < n$.
If

\begin{equation*}
b'_k = \left\{
\begin{aligned}
    b_{k}-1\ &\text{ if } l \le k \le r,\\
    b_k\ &\text{ otherwise},\\
\end{aligned}\right.
\end{equation*}
then we write $\boldsymbol{b}\xrightarrow{l,r}\boldsymbol{b'}$.
If

\begin{equation*}
b'_k = \left\{
\begin{aligned}
    b_{k}-2\ &\text{ if } 0 \le k \le r'<r,\\
    b_{k}-1\ &\text{ if } r' < k \le r,\\
    b_k\ &\text{ otherwise},
\end{aligned}\right.
\end{equation*}
then we write $\boldsymbol{b}\xrightarrow{0,r}\xrightarrow{0,r'}\boldsymbol{b'}$ (or $\boldsymbol{b}\xrightarrow{0,r'}\xrightarrow{0,r}\boldsymbol{b'}$).
Note that if

\begin{equation*}
b'_k = \left\{
\begin{aligned}
    b_{k}-2\ &\text{ if } 0 \le k \le r'<r,\\
    b_{k}-1\ &\text{ if } r' < k \le r,\\
    b_k\ &\text{ otherwise},
\end{aligned}\right.
\end{equation*}

then $\boldsymbol{b'}\in \mathbb{SD}_n$.

\begin{lemma}\label{lemma:hook_and_shifted}
Let $S,S'\in \mathcal{F}(S_n)$.
The following are equivalent.
\begin{itemize}
    \item[(1)] There exists a box $(i,j)\in S$ such that $S' = S\setminus h_S{(i,j)}$.
    \item[(2)] There exist $0\leq l\le r < n$ such that $\boldsymbol{sd}(S)\xrightarrow{l,r}\boldsymbol{sd}(S')$, or there exist $0 \le r'<r<n$ such that $\boldsymbol{sd}(S)\xrightarrow{0,r}\xrightarrow{0,r'}\boldsymbol{sd}(S')$.
\end{itemize}
\end{lemma}

Let us explain the key point of a proof of the lemma by using some examples. 
Let $S\in \mathcal{F}(S_n)$, and write $\boldsymbol{sd}(S)$ as $\boldsymbol{sd}(S)=[d_{0},\ldots,d_{n}]$ for $S \in \mathcal{F}(S_n)$.
Let us consider $(1)\Longrightarrow (2)$. If $h(S)\leq j$, then the removed hook $h_s(i,j)$ is of the form either (b) or (c) in Example \ref{ex:shifted_hooks}.
Thus, there exist $0\le l \le r < n$ such that $\boldsymbol{sd}(S)\xrightarrow{l,r}\boldsymbol{sd}(S')$.
For example, let $S$ be as in Example \ref{ex:shifted_hooks}, and let $S' = S\setminus h_{S}(2,6)$.
Note that the right-half of $S$ is an (ordinary) Young diagram. 
Removing the hook $h_s(i,j)$ of this form from $S$ naturally corresponds to removing a hook from the Young diagram (see \cite[Chapter 4]{Olsson}).
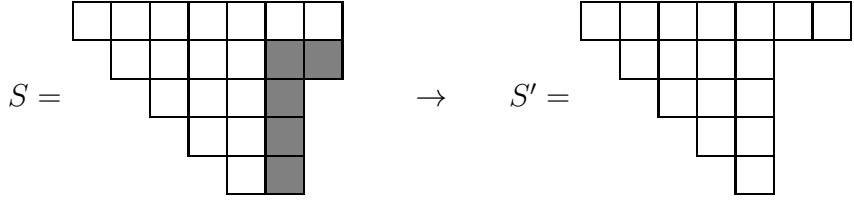
\begin{center}
$S=$\ \begin{ytableau}
    \ &\ &\ &\ &\ &\ &\ \\
    \none&\ &\ &\ &\ &*(gray) &*(gray) \\
    \none&\none&\ &\ &\ &*(gray) \\
    \none&\none&\none&\ &\ &*(gray) \\
    \none&\none&\none&\none&\ &*(gray) \\
\end{ytableau}
\qquad$\to$\qquad
$S'=\ $\begin{ytableau}
    \ &\ &\ &\ &\ &\ &\ \\
    \none&\ &\ &\ &\  \\
    \none&\none&\ &\ &\  \\
    \none&\none&\none&\ &\  \\
    \none&\none&\none&\none&\  \\
\end{ytableau}
\end{center}
In the diagonal expression, we see that 
\[
\boldsymbol{sd}(S) = [5,5,4,3,2,2,1,0],\qquad \boldsymbol{sd}(S') = [5,4,3,2,1,1,1,0],
\]
and hence $\boldsymbol{sd}(S)\xrightarrow{1,5}\boldsymbol{sd}(S')$.

If $j<h(S)$, then the removed hook is of the form (a) in Example \ref{ex:shifted_hooks}.
In this case, we deduce that $\boldsymbol{sd}(S)\xrightarrow{0,r}\xrightarrow{0,r'}\boldsymbol{sd}(S')$ for some $0 \le r' < r < n$.
For example, let $S$ be as in Example \ref{ex:shifted_hooks}, and let $S' = S\setminus h_{S}(2,3)$.


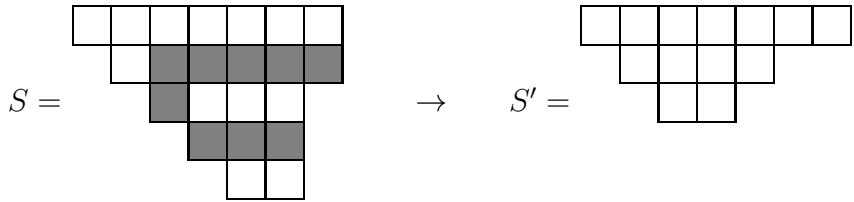
\begin{center}
$S=$\ \begin{ytableau}
    \ &\ &\ &\ &\ &\ &\ \\
    \none&\ &*(gray)\ &*(gray)\ &*(gray)\ &*(gray)\ &*(gray)\ \\
    \none&\none&*(gray)\ &\ &\ &\ \\
    \none&\none&\none&*(gray)\ &*(gray)\ &*(gray)\ \\
    \none&\none&\none&\none&\ &\ \\
\end{ytableau}
\qquad$\to$\qquad
$S'=\ $\begin{ytableau}
    \ &\ &\ &\ &\ &\ &\ \\
    \none&\ &\ &\ &\ \\
    \none&\none&\ &\ \\
    \none\\
    \none\\
\end{ytableau}\\
\end{center}
In the diagonal expression, we see that 
\[
\boldsymbol{sd}(S) = [5,5,4,3,2,2,1,0],\qquad \boldsymbol{sd}(S') = [3,3,2,2,1,1,1,0],
\]
and hence $\boldsymbol{sd}(S)\xrightarrow{0,5}\xrightarrow{0,2}\boldsymbol{sd}(S')$.

The implication $(2)\Longrightarrow (1)$ can be shown in exactly same as Lemma \ref{lemma:hook_and_DRep}.

\begin{definition}\label{def:symmetric}
A sequence $(a_{-m},\ldots,a_n)\in \mathbb{D}_{m,n}$ is said to be symmetric if $a_i=a_{n-m-i}$ for all $-m\leq i\leq n$.
\end{definition}

\begin{lemma}\label{lem:nn_symmetric}
\mbox{}
\begin{enumerate}\label{sy}
    \item[(1)] Let $Y\in\mathcal{F}(Y_{n,n})$. The sequence $\boldsymbol{d}(Y)\in \mathbb{D}_{n,n}$ is symmetric if and only if $Y\in\mathcal{T}(Y_{n,n})$.
    \item[(2)] Let $Y\in\mathcal{F}(Y_{n,n+1})$. The sequence $\boldsymbol{d}(Y)\in \mathbb{D}_{n,n+1}$ is symmetric if and only if $Y\in\mathcal{T}(Y_{n,n+1})$.
\end{enumerate}
\end{lemma}

\begin{proof}
By Theorem \ref{th:game_cor}, we need only to show part $(1)$, because it is clear that for $Y\in\mathcal{T}(Y_{n,n})$, $\boldsymbol{d}(Y)$ is symmetric if and only if $\boldsymbol{d}(E(Y))$ is symmetric.
We show by induction on $\#Y$ that if $Y\in \mathcal{T}(Y_{n,n})$, then $\boldsymbol{d}(Y)=(d_{-n}(Y),\ldots, d_{n}(Y))\in \mathbb{D}_{n,n}$ is symmetric.
If $Y=Y_{n,n}$, then $\boldsymbol{d}(Y_{n,n})=(0,1,\ldots,n-1,n,n-1,\ldots,1,0)$ is symmetric.
Assume that $Y\neq Y_{n,n}$.
There exists $\hat{Y}\in \mathcal{T}(Y_{n,n})$ such that $\hat{Y}\to Y$;
note that $d(\hat{Y})=(d_{-n}(\hat{Y}),\ldots, d_{n}(\hat{Y}))$ is symmetric by the induction hypothesis, and $d_{k-1}(\hat{Y})\nearrow d_{k}(\hat{Y})$ if and only if $d_{-k}(\hat{Y})\searrow d_{-k+1}(\hat{Y})$ for $-n < k\le n$.
Then,
\begin{itemize}
    \item[(i)] there exist $-n<l\leq r<n$ such that $\boldsymbol{d}(\hat{Y})\xrightarrow{l,r}\boldsymbol{d}(Y)$, or
    \item[(ii)] there exist $-n<l\leq r<n$ such that $\boldsymbol{d}(\hat{Y})\xrightarrow{l,r}\boldsymbol{d}(\hat{Y}')\xrightarrow{l'=-r,r'=-l}\boldsymbol{d}(Y)$.
\end{itemize}

Let us consider case (i). 
Suppose that $l \not= -r$.
Note that $d_{l-1}(\hat{Y})\nearrow d_{l}(\hat{Y})$, $d_{-l}(\hat{Y})\searrow d_{-l+1}(\hat{Y})$, $d_{r}(\hat{Y})\searrow d_{r+1}(\hat{Y})$, and $d_{-r-1}(\hat{Y})\nearrow d_{-r}(\hat{Y})$. 
By Lemma \ref{lemma:bulge_hook}, we have $d_{-r-1}(Y)\nearrow d_{-r}(Y)$ and $d_{-l}(Y)\searrow d_{-l+1}(Y)$.
Thus $\boldsymbol{d}(Y)_{[-r,-l]}\in \mathbb{D}_{n,n}$ by Lemma \ref{lemma:same_number_hook}, which is a contradiction. 
Hence we deduce that $l=-r$.
Then we have $\boldsymbol{d}(\hat{Y})\xrightarrow{-r,r}\boldsymbol{d}(Y)$; in this case, it is obvious that $\boldsymbol{d}(Y)\in \mathbb{D}_{n,n}$ is symmetric.

Let us consider case (ii).
We will show that $d_{k}(Y) = d_{-k}(Y)$ for any $-n < k < n$.
Assume that $l \le k \le r$ and $-r \le k \le -l$.
In this case, we have $d_{k}(\hat{Y}) = d_{k}(\hat{Y}')+1 = d_{k}(Y)+2$.
Because $l \le -k \le r$ and $-r \le -k \le -l$, we have $d_{-k}(\hat{Y}) = d_{-k}(\hat{Y}')+1 = d_{-k}(Y)+2$ .
Thus we have $d_{k}(Y) = d_{k}(\hat{Y})-2 = d_{-k}(\hat{Y})-2 = d_{-k}(Y)$.
The proofs for the other cases are similar.
Hence $\boldsymbol{d}(Y)\in \mathbb{D}_{n,n}$ is symmetric.

Next, we show that if $\boldsymbol{d}(Y)=(d_{-n}(Y),\ldots, d_{n}(Y))\in \mathbb{D}_{n,n}$ is symmetric, then $Y\in\mathcal{T}(Y_{n,n})$.
Let $\mathbb{A} := \{0\leq i\leq n-1\mid d_i=d_{i+1}+1\}$, and write it as: $\mathbb{A}=\{i_1,i_2,\ldots,i_k\}$.
Then there exists a transition $Y_{n,n}=Y_0\xrightarrow{}Y_1\xrightarrow{}Y_2\xrightarrow{}\cdots\xrightarrow{}Y_{k-1}\xrightarrow{} Y_k=Y$ such that $\boldsymbol{d}(Y_{l-1})\xrightarrow{-i_l,i_l}\boldsymbol{d}(Y_l)$ for $1\leq l\leq k$.
Thus we obtain $Y\in\mathcal{T}(Y_{n,n})$, as desired.
\end{proof}

Let $\boldsymbol{a}=(a_{-n},a_{n-1},\ldots,a_{-1},\dot{a_0},a_1,\ldots,a_{n},a_{n+1})\in \mathbb{D}_{n,n+1}$.
Assume that 
\[
\boldsymbol{\hat{a}}:=[a_1,a_2,\ldots,a_{n},a_{n+1}]\in\mathbb{N}_0^{n+1}.
\]
By the definition of $\mathbb{D}_{n,n+1}$, we have $\boldsymbol{\hat{a}}\in\mathbb{SD}_{n}$.
\begin{definition}\label{def:sh-di}
The map $A:\mathcal{T}(Y_{n,n+1}) \to \mathcal{F}(S_n)$ is defined as follows. 
If the diagonal expression of $Y \in \mathcal{T}(Y_{n,n+1})$ is
\[
\boldsymbol{d}(Y)=(a_{-n},a_{n-1},\ldots,a_{-1},\dot{a_0},a_1,\ldots,a_{n},a_{n+1}),
\]
then we define $A(Y)\in \mathcal{F}(S_n)$ to be the shifted Young diagram in $\mathcal{F}(S_n)$ whose diagonal expression is equal to
\[
\boldsymbol{sd}(A(Y)) = [a_1,a_2,\ldots,a_{n},a_{n+1}].
\]
\end{definition}

\begin{lemma}\label{lemma:nn+1_sn}
Let $Y \in \mathcal{T}(Y_{n,n+1})$, and let $Y'\in \mathcal{O}(Y)$.
Also, set $S := A(Y) \in \mathcal{F}(S_n)$.
Then there exists $S'\in \mathcal{O}(S)$ such that $A(Y') = S'$.
\end{lemma}
\begin{proof}
Since $Y'\in \mathcal{O}(Y)$, we see that
\begin{itemize}
    \item[(i)] there exist $-n<l\leq r<n+1$ such that $\boldsymbol{d}(Y)\xrightarrow{l,r}\boldsymbol{d}(Y')$, or
    \item[(ii)] there exist $-n<l\leq r<n+1$ and $Y''\in \mathcal{F}(Y_{n,n+1})$ such that $\boldsymbol{d}(Y)\xrightarrow{l,r}\boldsymbol{d}(Y'')\xrightarrow{-r+1,-l+1}\boldsymbol{d}(Y')$.
\end{itemize}

First, we consider case (i). 
By the proof of Lemma \ref{lem:nn_symmetric}, we see that $l=-r+1$, and hence $\boldsymbol{d}(Y)\xrightarrow{-r+1,r}\boldsymbol{d}(Y')$.
In this case, we have $d_{r-1}(S) = d_{r}(S)+1$.
Let $S'\in \mathcal{F}(S_n)$ be such that $\boldsymbol{sd}(S)\xrightarrow{0,r-1}\boldsymbol{sd}(S')$.
Then we deduce that $A(Y')=S'$.

Next, we consider case (ii). 
By the proof of Lemma \ref{lem:nn_symmetric}, we see that $l\not=-r+1$, and hence $\boldsymbol{d}(Y)_{[l,r]}, (\boldsymbol{d}(Y)_{[l,r]})_{[-r+1,-l+1]} \in \mathbb{D}_{n,n+1}$.

Assume that $0\leq l\leq r$.
In this case, we have $d_{l-2}(S) = d_{l-1}(S)$ and $d_{r-1}(S) = d_{r}(S)+1$.
Let $S'\in \mathcal{F}(S_n)$ be such that $\boldsymbol{sd}(S)\xrightarrow{l-1,r-1}\boldsymbol{sd}(S')$.
Then we deduce that $A(Y')=S'$.

Assume that $l\leq r\leq 0$.
In this case, we have $d_{-r-1}(S) = d_{-r}(S)$ and $d_{-l}(S) = d_{-l+1}(S)+1$.
Let $S'\in \mathcal{F}(S_n)$ be such that $\boldsymbol{sd}(S)\xrightarrow{-r,-l}\boldsymbol{sd}(S')$.
Then we deduce that $A(Y')=S'$.

Assume that $l\leq 0 < r$.
In this case, we have $d_{r-1}(S) = d_{r}(S)+1$ and $d_{-l}(S) = d_{-l+1}(S)+1$.
Let $S'\in \mathcal{F}(S_n)$ be such that $\boldsymbol{sd}(S)\xrightarrow{0,-l}\xrightarrow{0,r-1}\boldsymbol{sd}(S')$.
Then we deduce that $A(Y')=S'$.

Thus we have proved the lemma.
\end{proof}

Let $\boldsymbol{b}=[b_0,b_1,\ldots,b_{n-1},b_n]\in \mathbb{SD}_{n}$.
Assume that 
\[
\boldsymbol{\hat{b}}:=(b_{-n},b_{n-1},\ldots,b_{-1},\dot{b_0},b_0,b_1,\ldots,b_{n-1},b_n)\in\mathbb{N}_0^{2n+2}.
\]
By the definition of $\mathbb{SD}_{n}$, we have $\boldsymbol{\hat{b}}\in\mathbb{D}_{n,n+1}$.

\begin{definition}\label{def:B}
The map $B:\mathcal{F}(S_n) \to \mathcal{T}(Y_{n,n+1})$ is defined as follows. 
If the diagonal expression of $Y \in \mathcal{F}(S_n)$ is
\[
\boldsymbol{sd}(S)=[a_0,a_1,\ldots,a_{n-1},a_n].
\]
then we define $B(S) \in \mathcal{T}(Y_{n,n+1})$ to be the rectangular Young diagram in $\mathcal{T}(Y_{n,n+1})$ whose diagonal expression is equal to
\[
\boldsymbol{d}(B(S)) =(a_{n},a_{n-1},\ldots,\dot{a_0},
\underbrace{a_0}_{1\text{st}},a_1,\ldots,a_{n-1},
\underbrace{a_n}_{(n+1)\text{-th}}).
\]
\end{definition}

\begin{lemma}\label{lemma:sn_nn+1}
Let $S\in \mathcal{F}(S_n)$, and let $S'\in \mathcal{O}(S)$.
Also, set $Y := B(S) \in \mathcal{T}(Y_{n,n+1})$.
Then there exists $Y'\in \mathcal{O}(Y)$ such that $B(S') = Y'$.
\end{lemma}

\begin{proof}
Since $S'\in \mathcal{O}(S)$, we see that
\begin{itemize}
    \item[(i)] there exist $0\leq l\leq r<n$ such that $\boldsymbol{sd}(S)\xrightarrow{l,r}\boldsymbol{sd}(S')$, or
    \item[(ii)] there exist $0\leq r'< r<n$ such that $\boldsymbol{sd}(S)\xrightarrow{0,r}\xrightarrow{0,r'}\boldsymbol{sd}(S')$.
\end{itemize}

First, we consider case (i). 
Assume that $l=0$.
In this case, $d_{r}(S) = d_{r+1}(S)+1$.
Then, we have $d_{r+1}(B(S)) = d_{r+2}(B(S))+1$, $d_{-r-1}(B(S))+1 = d_{-r}(B(S))$, and hence $\boldsymbol{d}(B(S))_{[-r,r+1]} \in \mathbb{D}_{n,n+1}$ by Lemma \ref{lemma:bulge_hook}.
Let $Y'\in \mathcal{O}(Y)$ be such that $\boldsymbol{d}(Y)\xrightarrow{-r,r+1}\boldsymbol{d}(Y')$. Then we deduce that $B(S')=Y'$.
Assume that $0<l\leq r$. 
In this case, $d_{l-1}(S) = d_{l}(S)$ and $d_{r}(S) = d_{r+1}(S)+1$.
Then, we have $d_{l}(B(S)) = d_{l+1}(B(S))$, $d_{-l}(B(S)) = d_{-l+1}(B(S))$, $d_{r+1}(B(S)) = d_{r+2}(B(S))+1$, $d_{-r-1}(B(S))+1 = d_{-r}(B(S))$, and hence $\boldsymbol{d}(B(S))_{[l+1,r+1]}$, $(\boldsymbol{d}(B(S))_{[l+1,r+1]})_{[-r,-l]} \in \mathbb{D}_{n,n+1}$ by Lemma \ref{lemma:bulge_hook}.
Let $Y'\in \mathcal{O}(Y)$ be such that $\boldsymbol{d}(Y)\xrightarrow{l+1,r+1}\boldsymbol{d}(Y'')\xrightarrow{-r,-l}\boldsymbol{d}(Y')$. Then we deduce that $B(S')=Y'$.

Next, we consider case (ii). 
In this case, $d_{r}(S) = d_{r+1}(S)+1$ and $d_{r'}(S) = d_{r'+1}(S)+1$.
Then, we have $d_{r+1}(B(S)) = d_{r+2}(B(S))+1$, $d_{-r-1}(B(S))+1 = d_{-r}(B(S))$, $d_{r'+1}(B(S)) = d_{r'+2}(B(S))+1$, $d_{-r'-1}(B(S))+1 = d_{-r'}(B(S))$, and hence $\boldsymbol{d}(B(S))_{[-r',r+1]}, (\boldsymbol{d}(B(S))_{[-r',r+1]})_{[-r,r'+1]} \in \mathbb{D}_{n,n+1}$ by Lemma \ref{lemma:bulge_hook}.
Let $Y'\in \mathcal{O}(Y)$ be such that $\boldsymbol{d}(Y)\xrightarrow{-r',r+1}\boldsymbol{d}(Y'')\xrightarrow{-r,r'+1}\boldsymbol{d}(Y')$. Then we deduce that $B(S')=Y'$.

Thus we have proved the lemma.
\end{proof}

The next theorem follows from Lemmas \ref{lemma:nn+1_sn} and \ref{lemma:sn_nn+1}.

\begin{theorem}\label{th:nn+1_shifted}
For $n\in\mathbb{N}$, MHRG$(n,n+1)$ and HRG$(S_n)$ are isomorphic.
In particular, $\mathcal{G}(Y_{n,n+1})$ in MHRG$(n,n+1)$ is equal to $\mathcal{G}(S_n)$ in HRG$(S_n)$.
\end{theorem}

Combining Proposition \ref{pr:shifted}, Theorems \ref{th:game_cor}, and \ref{th:nn+1_shifted}, we obtain the following corollary.

\begin{corollary}\label{cor:nn_nn+1}
In MHRG$(n,n)$ (resp., MHRG$(n,n+1)$), the $\mathcal{G}$-value of the starting position $Y_{n,n}$ (resp., $Y_{n,n+1}$) is equal to 
\[
    \mathcal{G}(Y_{n,n})=\mathcal{G}(Y_{n,n+1})=\bigoplus_{1\leq k\leq n}^{} k.
\]
\end{corollary}

\begin{example}\label{ex:nn+1}
Assume that $n=3$. The $\mathcal{G}$-value of $Y_{3,4}=\Yvcentermath1\young(3321,2332,1233)$ is equal to $1\oplus2\oplus3=0$.
\end{example}

\appendix
\def\thesection{\Alph{section}}
\section*{Appendix}
\section{Proof of Proposition \ref{prop:DRep}}

\begin{lemma}\label{lemma:diagonal2}
Let $Y \in \mathcal{F}(Y_{m,n})$;
recall that $d_k = d_k(Y) =\#\{(i,j)\in Y\mid j-i=k\}$ for $k \in \mathbb{Z}$.
\begin{enumerate}
    \item[(1)] If $k>0$, then $0\le d_{k-1}-d_k\le1$.
    \item[(2)] If $k\le0$, then $0\le d_k-d_{k-1}\le1$.
\end{enumerate}
\end{lemma}
\begin{proof}
(1) Assume that $d_k = 0$.
Then, $(1,k+1)\notin Y$ by Lemma \ref{lemma:diagonal}, which implies that $(2,k+1)\notin Y$.
Hence, $d_{k-1} = \max\{\min\{i,j\} \mid (i,j)\in Y, j-i=k-1\}$ is equal to $0$ or $1$ by Remark \ref{remark:d_box}.
Thus we obtain $0 \le d_{k-1} - d_k = d_{k-1} \le 1$.
Assume $d_k > 0$. 
By Lemma \ref{lemma:diagonal}, it follows that $(d_k,d_k+k)\in Y $ and $(d_k+1,d_k+k+1)\notin Y$.
Then we have $(d_k,d_k+k-1)\in Y $ and $(d_k+2,d_k+k+1)\notin Y$.
Therefore $d_{k-1} = \max\{\min\{i,j\} \mid (i,j)\in Y, j-i=k-1\}$ is $d_k$ or $d_k+1$  by Remark \ref{remark:d_box}.
Thus we obtain $0 \le d_{k-1} - d_k \le 1$.

(2) The proof of (2) is similar to that of (1).\qedhere
\end{proof}

\begin{proof}[Proof of Proposition \ref{prop:DRep}]
From Lemma \ref{lemma:diagonal2}, the pair $(d_{i-1}(Y), d_{i}(Y))$ satisfies the adjacency condition for all $-m < i \le n$.
Since $d_{-m}(Y) = d_{n}(Y) = 0$, we have $(d_{-m}(Y),\ldots,d_{n}(Y)) \in \mathbb{D}_{m,n}$.

By the definition of $\boldsymbol{d}=\boldsymbol{d}_{m,n}$, it is obvious that $\boldsymbol{d}$ is an injection.

For $\boldsymbol{a} = (a_{-m},\ldots,a_n) \in \mathbb{D}_{m,n}$, we define $Y$ as follows.
The box $(i,j)$ is contained in $Y$ if and only if $\min\{i,j\} \le a_{j-i}$.
Note that if $j-i \le -m$ or $ n \le j-i$, then the box $(i,j)$ is not contained in $Y$.
We claim that $Y$ is a Young diagram.
It suffices to show that if the box $(i,j)$ is not contained in $Y$, then neither the box $(i+1,j)$ nor $(i,j+1)$ is contained in $Y$.
If $-m < j-i < 0$, then $\min\{i,j\} = j > a_{j-i}$.
By the definition of $\mathbb{D}_{m,n}$, we have $0 \le a_{j-i}-a_{j-i-1}$ and $a_{j-i+1}-a_{j-i} \le 1$.
Then we get $\min\{i+1,j\} = j > a_{j-i-1}$ and $\min\{i,j+1\} = j+1 > a_{j-i+1}$. 
Hence, by the definition of $Y$, we obtain $(i+1,j),(i,j+1) \notin Y$.
The proofs for the cases that $j-i = 0$ and $0 < j-i < n$ are similar.
Thus we have shown that $Y$ is a Young diagram.
Further, since $(m+1,1),(1,n+1) \notin Y$, it follows that $Y \in \mathcal{F}(Y_{m,n})$.

By the definition of $Y$, we have $d_k = a_k$ for $-m < k < n$.
Hence we obtain $\boldsymbol{d}(Y) = \boldsymbol{a}$, which shows that $\boldsymbol{d}$ is a surjection.
Thus we have proved that $\boldsymbol{d}$ is a bijection. \qedhere 
\end{proof}



\end{document}